\documentclass[10pt]{amsart} %

\usepackage{float}
\usepackage{tikz}
\usetikzlibrary{shapes.geometric, positioning}

\usepackage{epsfig}
\usepackage{hyperref}
\usepackage{array,amsmath, enumerate, fullpage}
\usepackage{amssymb, amsaddr}
\usepackage{color}

\newcommand{\Int}{{\rm Int}}
\newcommand{\Ext}{{\rm Ext}}

\usepackage{array}
\newcolumntype{C}[1]{>{\centering\let\newline\\\arraybackslash\hspace{0pt}}m{#1}}

\newcommand{\F}{\mathbb{F}}
\newcommand{\x}{\textbf{x}}

\newtheorem{theorem}{Theorem}[section]
\newtheorem{conj}{Conjecture}[section]
\newtheorem{lm}[theorem]{Lemma}

\title{Planar Graphs with Ore-degree at Most seven is strongly $13$-edge-colorable}

\author{Seth Nelson \and Gexin Yu}

\address{
Department of Mathematics, William \& Mary, Williamsburg, VA, 23185, USA.
}

\email{gyu@wm.edu}
\date{\today}

\begin{document}

\begin{abstract}
A strong edge-coloring of a graph $G$ is a coloring of edges of $G$ such that every color class forms an induced matching. The strong chromatic index is the minimum number of colors needed to color the graph. The Ore-degree $\theta(G)$ of a graph $G$ is the maximum sum of degrees of adjacent vertices. We show that every planar graph $G$ with $\theta(G)\le 7$ has strong chromatic index at most $13$. This settles a conjecture of Chen et al in the planar case. We use a discharging method, and apply Combinatorial Nullstellensatz to show reducible configurations. We provide an algorithm to allow Combinatorial Nullstellansatz  extracting coefficients from large polynomials.
\end{abstract}

\maketitle

\section{Introduction}

A strong edge-coloring of a graph is a proper edge-coloring in which any path or cycle of length three is assigned three distinct colors; that is, every color class is an induced matching. The {\em strong chromatic index} of a graph $G$ is the minimum number of colors required in a strong edge-coloring of $G$, and we denote this number as $\chi_s'(G)$. Strong edge-coloring was first introduced by Fouquet and Jolivet \cite{FJ83}. Erd\H{o}s and Ne\u{s}et\u{r}il \cite{EN86} made the following well-known conjecture.

\begin{conj}(Erd\H{o}s and Ne\u{s}et\u{r}il, 1986)
For a graph $G$ with maximum degree $\Delta$,
\begin{equation*}
\chi_s'(G) \le \begin{cases}
 \frac54 \Delta^2, & \text{if $\Delta$ is even}\\
\frac14 (5\Delta^2 - 2\Delta + 1), & \text{if $\Delta$ is odd}
\end{cases}
\end{equation*}
\end{conj}

This conjecture was confirmed when $\Delta = 3$ by Andersen \cite{A92} and independently by Hor\'ak, He, and Trotter \cite{HHT93}. For $\Delta=4$, Huang, Santana and G. Yu~\cite{HSY18}  showed that $\chi_s'(G)\le 21$, one more than the conjectured bound $20$. For large $\Delta$, the current best upper bound, provided by Hurley, Verclos, and Kang \cite{HVK22} using probabilistic techniques, is $\chi_s'(G) \le 1.772\Delta^2$.

If $G$ is a planar graph, then one may show that $\chi_s'(G) \le 4\Delta + 4$. Additionally, Faudree, Schelp, Gy\'arf\'as and Tuza~\cite{FSGT90} showed that for any $\Delta\ge 2$, there exists a planar graph $G$ of maximum degree $\Delta$ so that $\chi_s'(G) \ge 4\Delta - 4$.   Consequently, if $G$ is a planar graph with maximum $4$, $12\le \chi_s'(G)\le 20$. Yang, Shiu, Wang and Chen improved it to $\chi_s'(G)\le 19$, which was extended to a list version by Chen, Hu, X. Yu and Zhou~\cite{CHYZ19}.  

In this paper, we consider strong edge-colorings subject to the Ore-degree of a graph: $$\theta(G) = \max\{d(u) + d(v) : uv\in E(G)\}.$$  Note that $\Delta(G)+\delta(G)\le \theta(G)\le 2\Delta(G)$.  In the case when $\theta(G) = 6$, Nakprasit and Nakprasit \cite{NN15} have proved that $\chi_s'(G) \le 10$.  Chen, Huang, G. Yu and Zhou \cite{CHYZ20} showed that $\chi_s'(G) \le 15$ if  $\theta(G)\le 7$, and
made the following conjecture (the bound $13$ is optimal, see Figure~\ref{fig1}):

\begin{conj}\label{chen-conjecture}(Chen, Huang, Yu and Zhou \cite{CHYZ20})
If $G$ is a graph with $\theta(G) = 7$, then $\chi_s'(G) \le 13$.
\end{conj}

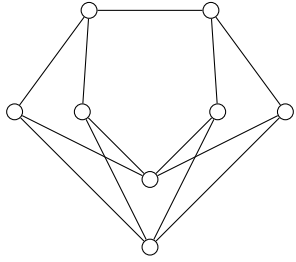
\begin{figure}[h]\label{fig1}
\begin{tikzpicture}[scale=0.9, every node/.style={font=\small}]

\tikzstyle{vertex}=[circle, draw, inner sep=1pt, minimum size=6pt]

\node[vertex] (v1) at (3,1) {};
\node[vertex] (v2) at (3,2) {};
\node[vertex] (v3) at (1,3) {};
\node[vertex] (v4) at (2,3) {};
\node[vertex] (v5) at (4,3) {};
\node[vertex] (v6) at (5,3) {};
\node[vertex] (v7) at (2.1,4.5) {};
\node[vertex] (v8) at (3.9,4.5) {};

\draw (v1) -- (v3);
\draw (v1) -- (v4);
\draw (v1) -- (v5);
\draw (v1) -- (v6);
\draw (v2) -- (v3);
\draw (v2) -- (v4);
\draw (v2) -- (v5);
\draw (v2) -- (v6);
\draw (v7) -- (v3);
\draw (v7) -- (v4);
\draw (v8) -- (v5);
\draw (v8) -- (v6);
\draw (v7) -- (v8);
\end{tikzpicture}
\caption{A graph with $\theta(G)=7$ and $\chi_s'(G)=13$}
\end{figure}

Huang, Liu, and Zhou \cite{HLZ23} made partial progress toward this conjecture by proving that any planar graph $G$ with $\theta(G)\le 7$ and $\Delta(G)\le 4$ is strong list edge-colorable in $14$ colors. This is a slightly stronger result than strong edge-coloring, which implies that $\chi_s'(G) \le 14$ for this class of graphs. Very recently,  Wang \cite{W25} showed the conjecture holds for graphs with $\theta(G)\le 7$ and maximum average degree less than $\frac{28}9$.

The main purpose of this article is to prove the following theorem, which gives a positive answer to Conjecture \ref{chen-conjecture} for planar graphs.

\begin{theorem}\label{main-theorem}
Let $G$ be a planar graph. If $\theta(G)\le 7$, then $\chi_s'(G) \le 13$.
\end{theorem}

We use a  discharging argument in our proof. To show the reducible configurations, we use the following Hall's Theorem.  

\begin{theorem}[Hall's Theorem, \cite{H35}]
Let $U$ be a set with subsets $A_1, \dots, A_n$. We may select a system of distinct representative for $A_1, \dots, A_n$ if and only if for all $k \in \{1, \dots, n\}$ and any selection of $k$ distinct sets $A_{b_1}, A_{b_2}, \dots, A_{b_k}$ for $b_1, \dots, b_k \in \{1, \dots, n\}$,
\begin{equation*}
\Big| \bigcup_{i \in [k]} A_{b_i} \Big| \ge k.
\end{equation*}
\end{theorem}

We will also use Combinatorial Nullstellensatz.

\begin{theorem}[Alon \cite{A99}, Combinatorial Nullstellensatz] \label{nullstellens}
Let $\F$ be an arbitrary field. Let $f(x_1, \dots, x_n)$ be a polynomial in $\F[x_1, \dots, x_n]$. Suppose that the degree $\deg f$ of $f$ equals $\sum_{i=1}^n k_i$ for nonnegative integers $k_i$, and the coefficient $x_1^{k_1}\dots x_n^{k_n}$ is nonzero. Then if $S_1, \dots, S_n$ are subsets of $\F$ with $|S_i| > k_i$, there exists $s_1 \in S_1, \dots, s_2 \in S_2$ so that $f(s_1, \dots, s_n) \neq 0$.
\end{theorem}


It is often difficult to apply the Combinatorial Nullstellensatz to large configurations, since extracting the relevant coefficients requires expanding high-degree polynomials. For products involving dozens of terms, the full expansion can already become unmanageable, and for configurations of the size considered in this paper, the resulting polynomial may contain billions of monomials. Such computations quickly exceed both the memory and runtime capacities of a typical computer.  To overcome this barrier, we develop an algorithm that efficiently extracts coefficients from polynomial products with more than ninety factors. The algorithm reduces both the number of operations and the memory required, allowing us to carry out expansions that would otherwise be computationally infeasible. This tool is essential in the proof of Theorem~\ref{main-theorem}, where handling extremely large polynomial configurations is unavoidable. We will discuss this algorithm in the next section. 

\section{Notation and a polynomial reduction algorithm}

\subsection{Notation}

For our purposes, $G$ will always be a simple graph, and $V(G)$, $E(G)$ will denote the vertex sets and edge sets of $G$ respectively. For two edges $e, f$, we say {\em $e$ sees $f$} if $e$ and $f$ are contained in a path or cycle of length at most three. We say a vertex $v$ is an $d$-vertex if $v$ has $d$-neighbors. We also say $v$ is an $d^+$-vertex or $d^-$-vertex if $v$ has at least $d$ neighbors or at most $d$ neighbors respectively. Likewise, for $k$-faces and $k$-cycles. 
For a coloring $\phi$ of $G$, and a vertex $v$ in $G$, $c_\phi(v)$ denotes the set of colors on the edges incident to $v$, and if $e \in E(G)$, then $c_\phi(e)$ denotes the color on $e$. If $\phi$ is understood, or if the proof works just as well on multiple different $\phi$, then we often just write $c(v)$ or $c(e)$. If $H$ is a subgraph of $G$, $e$ an edge in $H$, and $\phi$ a coloring of $G-H$, then $A_\phi(e)$ denotes the set of colors available to $e$ after applying $\phi$ to $G$.

\subsection{Polynomial Reduction Algorithm}

We will often consider the coefficients on the monomials in polynomials with the following form:
    \begin{equation*}
        f(\x) = \frac{\prod_{1 \le i < j \le n}(x_i-x_j)}{h(\x)},
    \end{equation*}
\noindent
where $h(\x)$ is a product of terms $x_s-x_r$ such that each $x_s-x_r$ is a term in $\prod_{1 \le i < j \le n}(x_i-x_j)$.

To apply Theorem \ref{nullstellens}, for some set of integers $d_1, \dots, d_k$, we want to show that there exists a monomial $c\cdot x_1^{k_1}\dots x_n^{k_n}$ in $f(\x)$ so that $c\not=0$ and $k_i < d_i$ for each $d_i$. By definition, $f(\x) = (x_{s_1} - x_{s_2})\dots (x_{s_k}-x_{s_\ell})$ for some set of integers $s_1, s_2, \dots, s_k, s_\ell$. To compute the product $f(\x)$ efficiently, we specify a greedy ordering on the terms of $f(\x)$, in which we select the $n$-th term $x_{s_i}-x_{s_j}$ if it is minimal in the sense that if $\ell_{s_i}$ and $\ell_{s_j}$ are respectively the numbers of appearances of $x_{s_i}$ and $x_{s_j}$ in the $n$ previous terms, then one of $d_{s_i}-\ell_{s_i}$ or $d_{s_j}-\ell_{s_j}$ is the smallest possible integer out of all other unselected terms.

Now, suppose that $x_{s_1} - x_{s_2}, \dots, x_{s_k}-x_{s_\ell}$ is a list of the terms after the greedy ordering. We recursively define the algorithm as follows. Set $p(\x) = 1$. On step $N$, pick the first $m$ terms $b_1, \dots, b_m$ in the list $x_{s_1} - x_{s_2}, \dots, x_{s_k}-x_{s_\ell}$ so that there is one $x_{s_i}$ in the $b_i$ which appears $\ell_{s_i} \ge d_{s_i}$ times. Replace $p(\x)$ by $b_1\dots b_m\cdot p(\x)$. Search through $p(\x)$, and set any monomial $c\cdot x_1^{k_1}\dots x_n^{k_n}$ to zero if the exponent $k_{s_i}$ on $x_{s_i}$ is larger than $d_{s_i}$. Delete $b_1, \dots, b_m$ from the list. If the resulting list is empty, terminate, otherwise repeat.

This algorithm will output the set of monomials in $f(\x)$ so that the exponent $k_i$ of $x_i$ in the monomial is strictly smaller than $d_i$. Using these techniques, the memory complexity of the computation is substantially reduced, and we have been able to compute polynomials with $80$ or more terms in the product, depending on how restrictive the set of integers $d_1, \dots, d_k$ is. In order to compute these polynomials, we used the computer algebra system Singular \cite{GPS08}.

\section{Structural lemmas}

From now on, we let $G$ be a plane graph with $\theta(G)\le 7$ and a minimal counterexample to the claim that $\chi_s'(G) \le 13$.

\begin{lm}\label{no-1-vertex}
For $v\in V(G)$, $2\le d(v)\le 4$.
\end{lm}

\begin{proof}
Since $\theta(G)\le 7$, we just need to show that there are no $1$-vertices and $5$-vertices in $G$. 

First suppose that $G$ contains a $1$-vertex $v$, which is adjacent to $u$. Delete $v$ from $G$ and apply a coloring $\phi$ to $G-v$. Then the edge $uv$ sees at most $12$ edges in $G-v$, so $|A_\phi(uv)| \ge 1$. Therefore, we can extend $\phi$ to a good coloring of $G$ by coloring $uv$ with the remaining $c \in A_\phi(uv)$, a contradiction. 

Suppose now that $v$ is a $5$-vertex in $G$ and assume that $v$ is adjacent to $u_1, u_2, u_3, u_4, u_5$. Since $\theta(G)\le 7$,  $d(u_i)\le 2$. Let $\phi$ be a good coloring of $G-v$, and apply $\phi$ to $G$, leaving $u_1v, \dots, u_5v$ uncolored. If any of the $u_i$ (say $u_1$) is adjacent to a $4^-$-vertex, then $|A_\phi(u_iv)| \ge 4$ for all $2 \le i \le 5$, and $|A_\phi(u_1v)| \ge 5$. Therefore, by Hall's Theorem, we may extend $\phi$ to a good coloring of $G$. Otherwise, suppose all $u_i$ are adjacent to two $5$-vertices. For each $u_i$, call this $5$-vertex $w_i$, and note that $w_i, w_j$ are never adjacent. Uncolor $u_iw_i$ for $i\in [5]$. If we ever have $A_\phi(u_iw_i) \cap A_\phi(u_jw_j) \neq \emptyset$, then apply the same color to $u_iw_i, u_jw_j$. Note that $|A_\phi(u_iw_i)| \ge 5$, so we may extend a coloring to all $u_iw_i$, and then we have $|A_\phi(u_iv)| \ge 5$, for each $u_iv$ sees a color twice. Therefore, $A_\phi(u_iw_i) \cap A_\phi(u_jw_j) = \emptyset$. If there is some $w_i, w_j$ such that $c(w_i) = c(w_j)$, then we must have $A_\phi(u_iw_i) \cap A_\phi(u_jw_j) \neq \emptyset$, so there is some color $c_1 \in c(w_i)$ not in $c(w_j)$ for each $w_i, w_j$. Let $c_1 \in c(w_1)$ so that $c_1 \not \in c(w_2)$. Color all $u_iw_i$ so that $|A_\phi(u_iv)| \ge 4$. We attempt to color $u_2w_2$ in $c_1$. If this is possible, we have $|A_\phi(u_iv)| \ge 3$ for all $2 < i \le 5$, and $|A_\phi(u_1v)| \ge 4$, so we may extend $\phi$ to a good coloring of $G$ by Hall's Theorem. Otherwise, some $u_iw_i$ is colored by $c_1$, so $|A_\phi(u_1v)| \ge 5$, and otherwise $|A_\phi(u_iv)| \ge 4$ for $i \in [5]\backslash \{1\}$, so again we may extend $\phi$ to a good coloring of $G$ by Hall's Theorem.
\end{proof}

\begin{lm}\label{2-vertex-adjacent-4-vertex}
Every $2$-vertex is adjacent to two $4$-vertices, and every $4$-vertex is adjacent to at most two $2$-vertices.
\end{lm}

\begin{proof}
Let $v$ be a $2$-vertex in $G$ with neighbors $u_1$ and $u_2$. Suppose that $u_1$ is a $2^+$-vertex and $u_2$ is a $3^-$-vertex. Delete $v$ from $G$ and let $\phi$ be a good coloring of $G-v$. Place $\phi$ on $G$, leaving $u_1v, u_2v$ uncolored. Then, $|A_\phi(u_1v)|, |A_\phi(u_2v)| \ge 2$, so we may color $u_1v$ and $u_2v$ in the two remaining colors.

Suppose that $v$ is a $4$-vertex in $G$ that is adjacent to three $2$-vertices. Let the neighbors of $v$ be $u_1, u_2, u_3, u_4$ with $d(u_1)=d(u_2)=d(u_3)=2$. Let $\phi$ be a good coloring of $G-v$, and place $\phi$ on $G$, leaving $vu_1, vu_2, vu_3, vu_4$ uncolored. Then, $|A_\phi(vu_1)|, |A_\phi(vu_2)|, |A_\phi(vu_3)| \ge 5$, and $|A_\phi(vu_4)| \ge 2$, so we may extend $\phi$ to a good coloring of $G$ by Hall's Theorem.
\end{proof}
	
\begin{lm}\label{no-3-cycles}
There are no $3$-cycles in $G$.
\end{lm}

\begin{proof}
Let $T = v_1v_2v_3v_1$ be a $3$-cycle. We can guarantee that all three vertices are $3^+$ vertices, for if one of $v_1, v_2, v_3$ is a $2$-vertex, the remaining vertices must be adjacent $4$-vertices by Lemma \ref{2-vertex-adjacent-4-vertex}, and this contradicts the Ore-degree restriction on $G$. We may suppose therefore that $v_2, v_3$ are $3$-vertices. Let $u_1, u_2$ be adjacent to $v_2, v_3$ respectively. Now, let $\phi$ be a good coloring of $G - \{v_1, v_2\}$. Then $|A_\phi(v_2u_1)|, |A_\phi(v_3u_2)| \ge 4$, $|A_\phi(v_1v_2)|, |A_\phi(v_3v_1)| \ge 5$, and $|A_\phi(v_2v_3)| \ge 9$. Therefore, we may extend $\phi$ to a good coloring of $G$ by coloring $vv_4, vu_1, vu_2, vu_3$ in order. 
	\end{proof}

\begin{lm}\label{no-2-vertex-4-cycle}
No $4$-cycle contains a $2$-vertex.
\end{lm}
\begin{proof}
Let $F=v_1v_2v_3v_4$ be a $4$-cycle with $d(v_2)=2$. Then Lemma~\ref{2-vertex-adjacent-4-vertex}, $d(v_1)=d(v_3)=4$ and $d(v_4)\le 3$.  Let $\phi$ be a good coloring of $G-v_2$. Then $|A_{\phi}(v_1v_2)|\ge 2$ and $|A_{\phi}(v_2v_3)|\ge 2$, so we can color $v_1v_2, v_2v_3$ to obtain a good coloring of $G$, a contradiction. 
\end{proof}

We now prove a collection of lemmas concerning vertex cuts and separating cycles. The general method is to consider a collection of vertices $S$ such that $G - S$ separates into two components $H_1, H_2$. We then independently color $G_1 = H_1 \cup S$ and $G_2 = H_2 \cup S$ using minimality. We glue $G_1$ and $G_2$ back together to make $G$, and we then consider the set of edges $E_1$ of $G_1$ and $E_2$ of $G_2$ so that each $e \in E_1$ sees some edge in $E_2$. We judiciously permute the colors of $G_1$ or $G_2$ to make sure no colors conflict among the edges of $E_1$ and $E_2$ to achieve a good coloring of $G$.

\begin{lm}\label{lem:no-adjacent-vertex-cut}
There is no vertex cut $\{u, v\}$ if $uv$ is an edge in $G$.
\end{lm}

\begin{proof}
 Suppose that $G$ contains a vertex cut $\{u, v\}$ with $uv\in E(G)$. Let $H_1, H_2$ be the distinct components of $G-\{u, v\}$, and let $G_1 = H_1 \cup \{u, v\}$ and $G_2 = H_2 \cup \{u, v\}$ be induced subgraphs of $G$. There are at most $6$ distinct edges incident to $u, v$, one of which is $uv$. There are at most $5$ other edges $f_1, f_2, f_3, f_4, f_5$ incident to $u, v$. We may assume that $G_1$ contains at most two of the $f_i$, so suppose $f_1, f_2$ are in $G_1$. Separately color $G_1, G_2$ in $\phi$ and $\psi$ by the minimality of $G$, and permute the color $c_\phi(uv)$ in $G_1$ to be the same as $c_\psi(uv)$ in $G_2$. If $c(f_1), c(f_2)$ are bad, then they share some color with $c(f_3), c(f_4), c(f_5)$. To remedy this situation, we will permute $c(f_1), c(f_2)$ such that they are not any of the three colors on $c(f_3), c(f_4), c(f_5)$. First, we cannot swap $c(f_1), c(f_2)$ with $c(uv)$, nor can we swap their colors with the colors on $c(f_3), c(f_4), c(f_5)$, so there are $4$ colors on $G_1$ we cannot swap with $c(f_1), c(f_2)$. We cannot swap $c(f_1)$ with itself, but presuming $c(f_1)$ is already bad on $f_1$, then it has been counted as one of the $c(f_3), c(f_4), c(f_5)$. Therefore, there are $4$ bad colors we do not swap $c(f_1)$ with, so we swap $c(f_1)$ with the $5$th available color. In the same manner, we then swap  $c(f_2)$ with the $6$th available color, taking care to only permute with colors on $G_1$. This gives a good coloring of $G_1 \cup G_2 = G$.
    \end{proof}

We now prove Lemmas \ref{lem:no-sep-4-cycle}, \ref{lem:no-sep-5-cycle}, \ref{lem:no-sep-6-cycle}, that there are no separating 4-, 5-, or 6-cycles. These lemmas are extremely important for our proofs using Combinatorial Nullstellens, which often requires a very large case analysis, unless one can guarantee that certain sets of vertices are not adjacent. For these next few proofs, we define the following piece of notation. Consider a cycle $C\subseteq G$. Recall that we have fixed a plane presentation of $G$. We define $\Int(C)$ to be the interior of $C$, including $C$ itself, and $\Ext(C)$ to be the exterior of $C$, including $C$ itself.

Our main tool is to split $G$ into $\Int(C)$ and $\Ext(C)$ and color $\Int(C)$ and $\Ext(C)$ using minimality. In the coloring of $\Int(C)$ and $\Ext(C)$, we color $C$ such that there is a color permutation of $\Int(C)$ which gives $C$ identical colors in both $\Int(C)$ and $\Ext(C)$. We glue $C$ in $\Int(C)$ and $\Ext(C)$ together, giving a coloring for all of $G$. By judicious permutation of colors, we extend this to a good coloring for all of $G$.

\begin{lm}\label{lem:no-sep-4-cycle}
There is no separating 4-cycle in $G$.
\end{lm}

 \begin{proof}
Suppose that $C$ is a separating $4$-cycle. We separately color the graphs of $\Int(C)$ and $\Ext(C)$. By Lemma \ref{lem:no-adjacent-vertex-cut}, $\Int(C)$ must contain at least two interior edges, and $\Ext(C)$ must contain at most four interior edges. If $C$ has two interior edges, they must not be adjacent to the same edge, so we split into two cases.

\noindent\textbf{Case 1:} $C$ has no interior edges incident to a $4$-vertex.

Say that $e, f$ are both non-cycle edges incident to $3$-vertices. We must have both $e, f$ as interior edges, otherwise we contradict Lemma~\ref{lem:no-adjacent-vertex-cut}. By minimality, we may color $\Int(C)$ and $\Ext(C)$ separately in $\phi$, $\psi$. Note that all colors on $C$ are distinct, so we permute the colors of $C$ in $\Int(C)$ to be identical to the colors of $C$ in $\Ext(C)$. Glue $\Int(C)$ and $\Ext(C)$ back together. We now permute the colos of $e, f$. When permuting $c_\phi(e), c_\phi(f)$, there are at most $4$ colors on the $4$-vertices $v_1, v_2$ which $c_\phi(e), c_\phi(f)$ cannot be permuted with. There are an additional $4$ colors on $C$ we cannot permute with. This leaves $4$ colors remaining, so we permute $c_\phi(e)$ to the $9$th color and $c_\phi(f)$ to the $10$th color. Note that no color on $C$ is changed by permutation of $c_\phi(e)$ or $c_\phi(f)$, for $e, f$ see all colors on $C$.

Note: as in Case 1, we will often require that the colors of the edges we permute in $\Int(C)$ are distinct from the colors on $C$. In this manner, we may isolate $\Int(C)$ and $\Ext(C)$ so that permutations do not cross over between the two different subgraphs. Occasionally, we may permit certain select, non-cycle edges to be colored identically to $C$. At these times, we will justify why this is permitted.

\vspace{0.25cm}
\noindent\textbf{Case 2:} $C$ has one interior edge incident to a $4$-vertex.

 Let $e_1, e_2$ both be incident to the same $4$-vertex, with $e_1$ in $\Int(C)$. Let $f_1, f_2$ be incident to $3$-vertices. We can guarantee both exist, otherwise Lemma \ref{lem:no-adjacent-vertex-cut} is violated. Color $\Int(C)$ and $\Ext(C)$ in $\phi, \psi$ respectively by minimality. All colors on $C$ are distinct in both graphs, so permute the colors on $\Int(C)$ such that we may glue $\Int(C)$ and $\Ext(C)$ back together. Note that no edge incident to $C$ shares a color on $C$, so we may safely permute the colors on our edges. First, permute $c_\phi(e_1)$ if it is bad. This edge sees at most $7$ colors in $\Ext(C)$, so we swap it with the $8$th available color. If $c_\phi(f_1), c_\phi(f_2)$ are bad, then we swap them as well. Both see at most $7$ colors in $\Ext(C)$. So, we are able to permute this to get good colors too. Finally, we may need to permute $c_\psi(e_2)$, in the case that $c_\psi(e_2)$ conflicts with some color on an endpoint of $e_1$. This edge sees at most $9$ colors in $\Int(C)$, so we swap with the $10$th available color.

\vspace{0.25cm}
\noindent\textbf{Case 3:} $C$ has at least two interior edges, both incident to a $4$-vertex. Note that $C$ may additionally have one interior edge incident to a $3$-vertex

 Let $e_1, e_2$ be two interior edges incident to a $4$-vertex, and let $f_1, f_2$ be two exterior edges incident to a $4$-vertex. Let $g$ be an interior edge incident to a $3$-vertex, if it exists. Color $\Int(C)$ and $\Ext(C)$ in $\phi$ and $\psi$ respectively using minimality. Note that every edge incident to $C$ receives a distinct color from those on $C$, so we need not worry about changing colors on $C$ when we permute $e_i$, $f_i$, or $g$. Now, every color on $C$ is distinct, so we may permute the colors of $\Int(C)$ so that each edge on $\Int(C)$ has the same color in $\Ext(C)$. Glue $\Int(C)$ and $\Ext(C)$ back together. Suppose $c_\phi(e_1) \neq c_\phi(e_2)$. Then, they cannot be swapped with the $4$ colors on the edges of $C$, and both see at most $5$ additional colors in $\Ext(C)$, for a total of $9$ colors. Thus, we have $4$ colors of space for $e_1, e_2$, and thus sufficient space to perform our swap. If $c_\phi(e_1) = c_\phi(e_2)$, then $c_\phi(e_1)$ may conflict with an additional two other colors, for a total of $11$, and we swap $c_\phi(e_1)$ with the $12$th available color. If $g$ exists, and $c_\phi(g)$ is bad, then we also permute $c_\phi(g)$. So, $c_\phi(g)$ may not be permuted with any color on $C$, nor any color in $\Ext(C)$ on the two possible additional edges which $g$ sees, nor any color on $e_1, e_2$. This is a total of $8$ colors, so we may again permute $c_\phi(g)$. Finally, in the case that $f_1, f_2$ are incident to the same vertices as $e_1, e_2$, we must permute $c_\psi(f_1), c_\psi(f_2)$, to not conflict with the colors of the edges incident to the endpoints of $e_1, e_2$. Through an identical method of counting for $e_1, e_2$, we are again able to perform this permutation.

Therefore, there are no separating $4$-cycles in $G$.
\end{proof}

The proofs of the following lemmas follow similar ideas to Lemma~\ref{lem:no-sep-4-cycle} but more tedious, so we put them in the appendix.  

\begin{lm}\label{lem:no-sep-5-cycle}
There are no separating $5$-cycles in $G$.
\end{lm}

\begin{lm}\label{lem:no-sep-6-cycle}
There is no separating $6$-cycle in $G$.
\end{lm}

\begin{lm}\label{4-cycle-two-4-vertices}
Every $4$-cycle contains two $4$-vertices and two $3$-vertices.
\end{lm}
\begin{proof}
Let $F = v_1v_2v_3v_4v_1$ be a $4$-cycle. We have already shown that all vertices on $F$ are $3^+$-vertices, by Lemma \ref{no-2-vertex-4-cycle}. So, we show that if $F$ contains only $3^+$-vertices, then two of those vertices are $4$-vertices. We may suppose that $v_1$ is a $4$ -vertex, since this contains the case where all vertices in $F$ are $3$-vertices. Therefore, let $v_1$ be adjacent to vertices $u_1, u_2$. Let every other vertex $v_i$ be adjacent to $u_{i+1}$. We label the edges as follows, $e_1 = v_1v_2, e_2 = v_2v_3, e_3 = v_3v_4, e_4 = v_4v_1, e_5 = v_1u_1, e_6 = v_1u_2, e_7 = v_2u_3, e_8 = v_3u_4, e_9 = v_4u_5$ (see Figure~\ref{fig309}).

\begin{figure}[h]
\begin{tikzpicture}[scale=1.5, every node/.style={font=\small}]

\tikzstyle{vertex}=[circle, draw, inner sep=1pt, minimum size=6pt]

\node[vertex] (v2) at (4,2) {$v_2$};
\node[vertex] (v3) at (4,1) {$v_3$};
\node[vertex] (v4) at (2,1) {$v_4$};
\node[vertex] (v1) at (2,2) {$v_1$};
\node[vertex] (u1) at (1,2) {$u_1$};
\node[vertex] (u2) at (2,3) {$u_2$};
\node[vertex] (u3) at (4,3) {$u_3$};
\node[vertex] (u4) at (5,1) {$u_4$};
\node[vertex] (u5) at (1,1) {$u_5$};

\draw (v1) -- (v2)  node[midway, above]{$e_1$};
\draw (v2) -- (v3) node[midway, right] {$e_2$};
\draw (v3) -- (v4) node[midway, above] {$e_3$};
\draw (v4) -- (v1) node[midway, right] {$e_4$};
\draw (v1) -- (u1) node[midway, above] {$e_5$};
\draw (v1) -- (u2) node[midway, right] {$e_6$};
\draw (v2) -- (u3) node[midway, right] {$e_7$};
\draw (v3) -- (u4) node[midway, above] {$e_8$};
\draw (v4) -- (u5) node[midway, above] {$e_9$};
\end{tikzpicture}
\caption{Figure for Lemma~\ref{4-cycle-two-4-vertices}}
\label{fig309}
\end{figure}
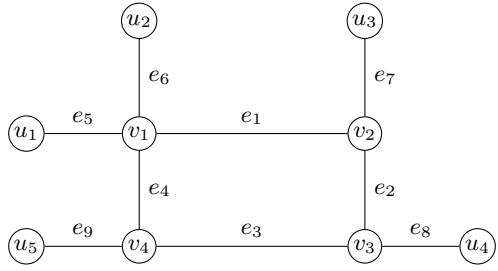

 We claim that $u_2u_4$ is not an edge and that $u_2 \neq u_4$. So, suppose that either $u_2u_4$ is an edge or that $u_2 = u_4$. Then, we have a cycle $C$ formed either by $u_2v_1v_2v_3u_4u_2$ or $u_2v_1v_2v_3u_4$ depending on $u_2 = u_4$ or not. Note that $C$ is either a $4$- or $5$-cycle. Since either $u_2u_4 \in E(G)$ or $u_2 = u_4$, then by planarity we may assume that $u_2, u_4$ are in $\Ext(F)$. If $u_3$ is in $\Int(C)$, then $C$ separates $u_3$ and $u_5$, contradicting Lemma \ref{lem:no-sep-4-cycle} or Lemma \ref{lem:no-sep-5-cycle}. Otherwise, we must have $u_3$ in $\Ext(C)$ and thus in $\Int(F)$, separating $u_3$ from $u_2$, and so $F$ is a separating $4$-cycle contradicting Lemma \ref{lem:no-sep-4-cycle}. In both cases, we achieve a contradiction, and so  $u_2u_4$ is not an edge and that $u_2 \neq u_4$, as claimed.

So, $u_2u_4$ is not an edge. Then, delete $F$ from $G$ and add $u_2u_4$ to the graph $G-F$. Call this new graph $G'$. Moreover, $G'$ is strictly smaller than $G$, so we may color $G'$ by $\phi$. In the coloring $\phi$, we have $|c(u_2)| = 2$, so set $c(u_2) = \{c_1, c_2\}$. Finally, let $c_3$ be the color on the edge $u_2u_4$. Now, apply $\phi$ to the graph $G$, leaving the edges $e_1, \dots, e_9$ uncolored. Note that $c_3\not\in A_\phi(e_8)$, for $u_2u_4$ has the color $c_3$ in the coloring $\phi$. Therefore, we apply $c_3$ to $e_8$.

	\vspace{0.25cm}
\noindent\textbf{Case 1:} 
No edge adjacent to $e_5$ is colored in $c_3$.

We color $e_6$ by $c_3$. Then, $|A_\phi(e_5)| \ge 2$, $|A_\phi(e_7)|, |A_\phi(e_9)| \ge 3$, $ |A_\phi(e_4)|, |A_\phi(e_1)| \ge 5$, and $|A_\phi(e_2)|, |A_\phi(e_3)| \ge 6$. We must now consider multiple subcases based on whether or not the vertices $u_3, u_5$ are incident to edges colored in $c_1$ or $c_2$.
		
			\vspace{0.25cm}
\noindent\textbf{Subcase 1.1:} 
Either $u_3$ sees $c_1$ and $c_2$, and $u_5$ sees at least one of $c_1, c_2$, or $u_3$ sees at least one of $c_1, c_2$, and $u_5$ sees $c_1$ and $c_2$.

The cases are symmetric, so we suppose $c_1, c_2 \in c(u_3)$, and either $c_1\in c(u_5)$ or $c_2 \in c(u_5)$. Then, $|A_\phi(e_5)| \ge 2$, $|A_\phi(e_7)|, |A_\phi(e_9)| \ge 3$, $|A_\phi(e_2)|, |A_\phi(e_3)| \ge 6$, $ |A_\phi(e_4)| \ge 6$, and $|A_\phi(e_1)| \ge 7$. So we may extend $\phi$ to a good coloring.

			\vspace{0.25cm}
\noindent\textbf{Subcase 1.2:} 
At most one of $c_1, c_2$ is in $c(u_3)$ and at least one of $c_1, c_2$ is in $c(u_5)$, or $c(u_5)$ contains at at least one of $c_1, c_2$ and at least one of $c_1, c_2$ is in $c(u_3)$.
			
Again, the cases are symmetric, so suppose that $u_3$ sees at most one of $c_1, c_2$ and $u_5$ sees at least one of $c_1, c_2$. Without loss of generality, let $c_1 \in c(u_5)$. Either $c_1 \not\in c(u_3)$ or $c_2\not\in c(u_3)$. Suppose first that $c_1 \not\in c(u_3)$. Then, color $e_2$ by $c_1$ so that $|A_\phi(e_5)|, |A_\phi(e_7)| \ge 2$, $|A_\phi(e_9)| \ge 3$, $|A_\phi(e_1)|, |A_\phi(e_3)| \ge 5$, and $|A_\phi(e_4)| \ge 6$. Then we may extend $\phi$ to a good coloring of $G$. Otherwise, if we can only color $e_2$ by $c_2$, then $|A_\phi(e_5)|, |A_\phi(e_7)|, |A_\phi(e_9)| \ge 2$, $|A_\phi(e_1)|, |A_\phi(e_3)| \ge 5$, and $|A_\phi(e_4)| \ge 6$. We may extend $\phi$ to a good coloring by coloring in order $e_7, e_5, e_9, e_1, e_3, e_4$.
			
			\vspace{0.25cm}
\noindent\textbf{Subcase 1.3:} 
Neither $u_3$ nor $u_5$ see $c_1, c_2$.

First, we mention that $A_\phi(e_7) \cap A_\phi(e_9) = \emptyset$. Indeed, suppose there exists $c_4$ in the intersection $A_\phi(e_7) \cap A_\phi(e_9)$. Color $e_7$ and $e_9$ in $c_4$. Thus,  $|A_\phi(e_5)| \ge 1$, $ |A_\phi(e_4)|, |A_\phi(e_1)| \ge 4$, and $|A_\phi(e_2)|, |A_\phi(e_3)| \ge 5$, so we may extend $\phi$ to a good coloring of $G$. Now, apply $c_1, c_2$ to $e_2, e_3$. Then $|A_\phi(e_5)| \ge 2$, $|A_\phi(e_1)|, |A_\phi(e_4)| \ge 5$, and either $|A_\phi(e_7)| \ge 2$ and $|A_\phi(e_9)| \ge 1$ or $|A_\phi(e_7)| \ge 1$ and $|A_\phi(e_9)| \ge 2$. Therefore, we color in the order $e_9, e_5, e_7, e_1, e_4$ or $e_7, e_5, e_9, e_1, e_4$ depending on $|A_\phi(e_9)| < |A_\phi(e_7)|$ or $|A_\phi(e_7)| < |A_\phi(e_9)|$.

			\vspace{0.5cm}
	\noindent\textbf{Case 2:} There exists an edge adjacent to $e_5$ colored in $c_3$.
	
In this case, we cannot apply $c_3$ to $e_6$. We first show that $A_\phi(e_7) \cap A_\phi(e_9) = \emptyset$, so suppose the converse holds. Then, there exists $c_4 \in A_\phi(e_7) \cap A_\phi(e_9)$, so apply $c_4$ to $e_7$ and $e_9$. If either $u_3$ or $u_5$ sees one of $c_1, c_2$, then $|A_\phi(e_5)|, |A_\phi(e_6)| \ge 2$, $|A_\phi(e_2)|, |A_\phi(e_3)| \ge 5$, and either $|A_\phi(e_4)|\ge 5$ and $|A_\phi(e_1)| \ge 6$ or $|A_\phi(e_4)|\ge 6$ and $|A_\phi(e_1)| \ge 5$. In either case we extend $\phi$ to a good coloring of $G$. Alternatively, suppose only one of $u_3$ or $u_5$ see $c_1$ or $c_2$. Without loss of generality, let $c_1 \in c(u_3)$ and $c_1, c_2 \not\in c(u_5)$. Then, color $e_3$ in $c_1$. We have $|A_\phi(e_5)|, |A_\phi(e_6)| \ge 2$, $|A_\phi(e_2)| \ge 4$, and $|A_\phi(e_2)|, |A_\phi(e_3)| \ge 5$, so we may extend $\phi$ to a good coloring of $G$ once again. Thus, $A_\phi(e_7) \cap A_\phi(e_9) = \emptyset$ as claimed. We now cover a variety of different subcases, which are more or less identical versions of the above.
	
			\vspace{0.25cm}
\noindent\textbf{Subcase 2.1:} Either $u_3$ sees $c_1$ and $c_2$, and $u_5$ sees at least one of $c_1, c_2$, or $u_3$ sees at least one of $c_1, c_2$, and $u_5$ sees $c_1$ and $c_2$.
			
By symmetry, we prove only the first case. Then, $|A_\phi(e_4)|, |A_\phi(e_5)|, |A_\phi(e_6)|, |A_\phi(e_9)| \ge 3$, $|A_\phi(e_2)|, |A_\phi(e_3)| \ge 6$, $|A_\phi(e_4)| \ge 7$, and $|A_\phi(e_1)| \ge 8$. Therefore, by coloring in the order $e_4, e_5, e_6, e_9, e_2, e_3, e_4, e_1$, we may extend $\phi$ to a good coloring of $G$.
			
			\vspace{0.25cm}
\noindent\textbf{Subcase 2.2:} Either $u_3$ sees at least one of $c_1, c_2$ and $u_5$ sees at most one of $c_1, c_2$, or $u_5$ sees at least one of $c_1, c_2$ and $u_3$ sees at least one.
			
Again, by symmetry, we prove only the first case. Without loss of generality. Suppose that $u_3$ sees $c_1$. Then, we may color $e_3$ in one of $c_1$ or $c_2$. In either case, $|A_\phi(e_7)|, |A_\phi(e_9)| \ge 2$, $|A_\phi(e_5)|, |A_\phi(e_6)| \ge 3$, $|A_\phi(e_2)| \ge 5$, $|A_\phi(e_4)| \ge 6$, and $|A_\phi(e_1)| \ge 7$. We may color $e_5, e_7, e_7, e_9$, since $A_\phi(e_7) \cap A_\phi(e_9) = \emptyset$, leaving $|A_\phi(e_2)| \ge 1$, $|A_\phi(e_4)| \ge 2$, and $|A_\phi(e_1)| \ge 3$, so we may extend $\phi$ to a good coloring of $G$. 
			
			\vspace{0.25cm}
\noindent\textbf{Subcase 2.3:} Neither $u_3$ nor $u_5$ see $c_1, c_2$.
			
Then, apply $c_1, c_2$ to $e_2, e_3$. Since $A_\phi(e_7) \cap A_\phi(e_9) = \emptyset$. After coloring $e_2, e_3$, we have $|A_\phi(e_7) \cup A_\phi(e_9)| \ge 4$, $|A_\phi(e_5)|, |A_\phi(e_6)| \ge 3$, and $|A_\phi(e_1)|, |A_\phi(e_4)| \ge 6$, so we may extend $\phi$ to a good coloring of $G$.

Thus, we have shown that $G$ can be colored in every case, so the Lemma is proven.
\end{proof}

\begin{lm}\label{no-adjacent-4-faces}
No $4$-face shares an edge with another $4$-face.
\end{lm}
	\begin{proof}
	Let $F_1 = v_1v_2v_3v_4v_1$ and $F_2 = v_2v_3v_5v_6v_2$. By Lemma \ref{4-cycle-two-4-vertices}, both $4$-faces must have two $4$-vertices, forcing either $v_2, v_4, v_5$ or $v_1, v_3, v_6$ to be $4$-vertices. By symmetry, suppose $v_1, v_3, v_6$ are all $4$-vertices, and that $v_2, v_4, v_5$ are all $3$-vertices. Finally, let $u_1, u_2$ be adjacent to $v_1$, let $u_3, u_4$ be adjacent to $v_6$, and let $u_5, u_6, u_7$ be adjacent to $v_3, v_4, v_5$ respectively. Call the edges $e_1 = v_1v_2, e_2 = v_2v_3, e_3 = v_3v_4, e_4 = v_4v_1, e_5 = v_3v_5, e_6 = v_5v_6, e_7 = v_2v_6, e_8 = v_1u_1, e_9 = v_1u_2, e_{10} = v_6u_3, e_{11} = v_6u_4, e_{12} = v_3u_5, e_{13} = v_4u_6, e_{14} = v_5u_7$ (see Figure~\ref{fig310}). 

\begin{figure}[h]
\begin{tikzpicture}[scale=1.5, every node/.style={font=\small}]

\tikzstyle{vertex}=[circle, draw, inner sep=1pt, minimum size=6pt]

\node[vertex] (v1) at (2,2) {$v_1$};
\node[vertex] (v2) at (3,2) {$v_2$};
\node[vertex] (v3) at (3,1) {$v_3$};
\node[vertex] (v4) at (2,1) {$v_4$};
\node[vertex] (v5) at (4,1) {$v_5$};
\node[vertex] (v6) at (4,2) {$v_6$};
\node[vertex] (u1) at (1,2) {$u_1$};
\node[vertex] (u2) at (2,3) {$u_2$};
\node[vertex] (u3) at (4,3) {$u_3$};
\node[vertex] (u4) at (5,2) {$u_4$};
\node[vertex] (u5) at (3,0) {$u_5$};
\node[vertex] (u6) at (1,1) {$u_6$};
\node[vertex] (u7) at (5,1) {$u_7$};

\draw (v1) -- (v2)  node[midway, below]{$e_1$};
\draw (v2) -- (v3) node[midway, right] {$e_2$};
\draw (v3) -- (v4) node[midway, above] {$e_3$};
\draw (v4) -- (v1) node[midway, right] {$e_4$};
\draw (v3) -- (v5) node[midway, above] {$e_5$};
\draw (v5) -- (v6) node[midway, right] {$e_6$};
\draw (v6) -- (v2) node[midway, below] {$e_7$};
\draw (v1) -- (u1) node[midway, above] {$e_8$};
\draw (v1) -- (u2) node[midway, left] {$e_9$};
\draw (v6) -- (u3) node[midway, left] {$e_{10}$};
\draw (v6) -- (u4) node[midway, above] {$e_{11}$};
\draw (v3) -- (u5) node[midway, left] {$e_{12}$};
\draw (v4) -- (u6) node[midway, above] {$e_{13}$};
\draw (v5) -- (u7) node[midway, above] {$e_{14}$};
\end{tikzpicture}
\caption{Figure for Lemma~\ref{no-adjacent-4-faces}}
\label{fig310}
\end{figure}
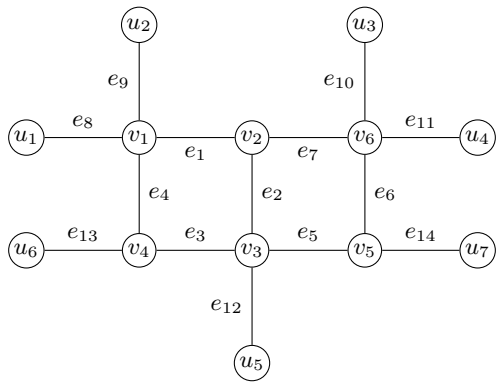
    
	Now, construct the graph $G'$ by adding the edge $v_1u_7$ to $G - F_2 - \{v_4\}$. It is possible to add such an edge, for no edge $v_1u_7$ exists in $G$, otherwise the Ore-degree is violated. Moreover, if $v_1 = u_7$, then by planarity there is no edge $v_6u_6$, so we complete the proof switching $v_6$ for $v_1$ and $u_7$ for $u_6$.  Thus, by the minimality of $G$, we color $G'$ in $\phi$, and let $c_1$ be the color on $v_1u_7$. Apply $\phi$ to $G$, leaving every edge except $e_8, e_9$ uncolored. We may apply $c_1$ to $e_1$ and $e_{14}$. In which case, $|A_\phi(e_{13})| \ge 1, |A_\phi(e_{10})|, |A_\phi(e_{11})| \ge 2$, $|A_\phi(e_4)| \ge 3$, $|A_\phi(e_{12})| \ge 4$, $|A_\phi(e_3)|, |A_\phi(e_6)| \ge 5$, $|A_\phi(e_6)| \ge 6$, $|A_\phi(e_5)| \ge 7$, $|A_\phi(e_2)| \ge 8$.

    Note now that $e_4$ does not see $e_6$, for then $u_6 \in \{v_6, v_5\}$ or, symmetrically, $u_7 \in \{v_1, v_4\}$. Whatever the case, either Lemma \ref{no-3-cycles} is violated, or Lemma \ref{lem:no-sep-4-cycle} is violated. Then, color $e_{13}, e_{10}$, and $e_{11}$ in that order. This is possible, for although $|A_\phi(e_{13})| \ge 1$ and $|A_\phi(e_{10})|, |A_\phi(e_{11})| \ge 2$, if $e_{13}$ sees one of $e_{10}, e_{11}$, then either $u_6 \in \{u_3, u_4\}$, $u_6u_4$ or $u_6u_3$ is an edge, or $u_6 = v_6$. In every case, either $e_{10}$ or $e_{11}$ has one extra color, or there is one less edge to color. Note that if any of $e_3, e_4$, or $e_{12}$ see $e_{10}, e_{11}$, then we have undercounted the number of colors available to them. So, either way, after coloring $e_{13}, e_{10}$, and $e_{11}$, we have $|A_\phi(e_4)| \ge 2$,  $|A_\phi(e_{12})| \ge 3$, and $|A_\phi(e_3)| \ge 4$. In addition, $|A_\phi(e_6)| \ge 3$, $|A_\phi(e_5)|, |A_\phi(e_7)| \ge 4$, and $|A_\phi(e_2)| \ge 5$. 
		
As before, let the colors on $e_8, e_9$ be $c_2, c_3$. Note that $c_2, c_3 \not \in c(u_7)$, so our proof technique is to attempt to color $e_5, e_6$ using the colors $c_2, c_3$. We consider four different cases.
	
	\vspace{0.25cm}
	\noindent\textbf{Case 1:} We may color $e_5, e_6$ in $c_2, c_3$.
	
	In this case, $|A_\phi(e_{12})| \ge 1$, $|A_\phi(e_4)| \ge 2$, $|A_\phi(e_3)|, |A_\phi(e_7)| \ge 4$, and $|A_\phi(e_2)| \ge 5$. So, we may extend $\phi$ to a good coloring.
	
	\vspace{0.25cm}
	\noindent\textbf{Case 2:} Both $e_5, e_6$ can be colored in at most one of $c_2, c_3$.
	
	In this case, $c(u_5)$ and $c(v_6)$ must both contain one of $c_2, c_3$. So, $|A_\phi(e_4)| \ge 2$, $|A_\phi(e_6)|, |A_\phi(e_{12})| \ge 3$, $|A_\phi(e_5)| \ge 4$, $|A_\phi(e_3)|, |A_\phi(e_7)| \ge 5$, and $|A_\phi(e_2)| \ge 6$. Coloring $e_5$ in the remaining color leaves $|A_\phi(e_4)|, |A_\phi(e_6)|, A_\phi(e_{12})| \ge 2$, $|A_\phi(e_3)|, |A_\phi(e_7)| \ge 5$, and $|A_\phi(e_2)| \ge 6$. Since $e_4$ does not see $e_6$,  we color $e_6$ to leave $A_\phi(e_{12})| \ge 1$, $|A_\phi(e_4)| \ge 2$, $|A_\phi(e_3)|, |A_\phi(e_7)| \ge 4$, and $|A_\phi(e_2)| \ge 5$, in which case we may extend $\phi$ to a good coloring by Hall's Theorem. Otherwise, we color $e_6$ in the remaining color, which leaves $|A_\phi(e_4)|, |A_\phi(e_{12})| \ge 2$, $|A_\phi(e_5)| \ge 3$, $|A_\phi(e_3)|, |A_\phi(e_7)| \ge 5$, and $|A_\phi(e_2)| \ge 6$, which again may be colored by Hall's Theorem. In both cases, we may extend $\phi$ to a good coloring of $G$.
	
	\vspace{0.25cm}
	\noindent\textbf{Case 3:} $e_6$ cannot be colored by $c_2$ or $c_3$.

	Then, we must have $c_2, c_3 \in c(v_6)$, so $e_5$ cannot be colored by either $c_1$ or $c_2$. Then, $|A_\phi(e_4) \cup A_\phi(e_6)| \ge 5$, $|A_\phi(e_{12})| \ge 3$, $|A_\phi(e_3)|, |A_\phi(e_5)| \ge 4$, $|A_\phi(e_7)| \ge 6$, and $|A_\phi(e_2)| \ge 7$. So, we may extend $\phi$ to a good coloring of $G$.
	
	\vspace{0.25cm}
	\noindent\textbf{Case 4:} $e_5$ cannot be colored by $c_2$ or $c_3$.
	
	We have already given a proof when $c_2, c_3 \in c(v_6)$, so suppose that $c_2, c_3\in c(u_5)$. So, we must not have one of $c_2, c_3$ in $c(v_6)$. Suppose that $c_2 \not \in c(v_6)$. Then, color $e_6$ in $c_2$. Then, $|A_\phi(e_4)| \ge 2$, $|A_\phi(e_5)|, |A_\phi(e_{12})| \ge 3$, $|A_\phi(e_7)| \ge 4$, and $|A_\phi(e_2)|, |A_\phi(e_3)| \ge 7$, so we may extend $\phi$ to a good coloring of $G$ by Hall's Theroem.

	We can color $G$ in every case, so no two $4$-faces share an edge.
	\end{proof}

\begin{lm}\label{no-2-vertex-5-face} There is no $2$-vertex on a $5$-face.
\end{lm}

\begin{proof}
By the Ore-degree condition and Lemma \ref{2-vertex-adjacent-4-vertex}, a $5$-face $F$ with a $2$-vertex must have two $4$-vertices and two $3$-vertices. Let $F = v_1v_2v_3v_4v_5v_1$. Suppose that $v_1$ is a $2$-vertex. Then, $v_2, v_5$ must be $4$-vertices and $v_3, v_4$ must be $3$-vertices. Let $v_5$ see $u_1, u_2$, let $v_2$ see $u_3, u_4$, let $v_3$ see $u_5$, and let $v_4$ see $u_6$. Denote the edges of $F$ by $e_1 = v_1v_2, e_2 = v_2v_3, e_3 = v_3v_4, e_4 = v_4v_5, e_5 = v_5v_1, e_6 = v_5u_1, e_7 = v_5u_2, e_8 = v_2u_3, e_9 = v_2u_4, e_{10} = v_3u_5, e_{11} = v_4u_6$ (see Figure~\ref{fig311}).

\begin{figure}[h]
\begin{tikzpicture}[scale=1.5, every node/.style={font=\small}]

\tikzstyle{vertex}=[circle, draw, inner sep=1pt, minimum size=6pt]

\node[vertex] (v1) at (3,2) {$v_1$};
\node[vertex] (v2) at (4,2) {$v_2$};
\node[vertex] (v3) at (4,1) {$v_3$};
\node[vertex] (v4) at (2,1) {$v_4$};
\node[vertex] (v5) at (2,2) {$v_5$};
\node[vertex] (u1) at (1,2) {$u_1$};
\node[vertex] (u2) at (2,3) {$u_2$};
\node[vertex] (u3) at (4,3) {$u_3$};
\node[vertex] (u4) at (5,2) {$u_4$};
\node[vertex] (u5) at (5,1) {$u_5$};
\node[vertex] (u6) at (1,1) {$u_6$};

\draw (v1) -- (v2)  node[midway, above]{$e_1$};
\draw (v2) -- (v3) node[midway, right] {$e_2$};
\draw (v3) -- (v4) node[midway, above] {$e_3$};
\draw (v4) -- (v5) node[midway, right] {$e_4$};
\draw (v5) -- (v1) node[midway, above] {$e_5$};
\draw (v5) -- (u1) node[midway, above] {$e_6$};
\draw (v5) -- (u2) node[midway, right] {$e_7$};
\draw (v2) -- (u3) node[midway, right] {$e_8$};
\draw (v2) -- (u4) node[midway, above] {$e_9$};
\draw (v3) -- (u5) node[midway, above] {$e_{10}$};
\draw (v4) -- (u6) node[midway, above] {$e_{11}$};
\end{tikzpicture}
\caption{Figure for Lemma~\ref{no-2-vertex-5-face}}
\label{fig311}
\end{figure}
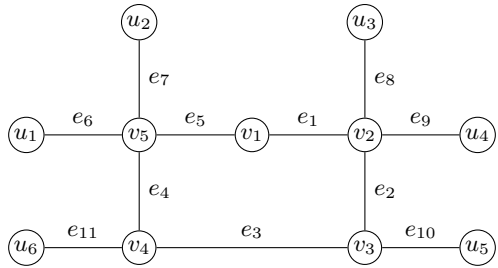

First, $e_{10}$ does not see $e_5$, for otherwise either $u_5 = v_5$ or $u_5 \in \{u_1, u_2\}$. The first case cannot hold by Lemma \ref{no-3-cycles}, and the second case cannot hold, for then $u_6$ must be contained in the interior of $u_5v_5v_4v_2u_5$, so $u_5v_5v_4v_2u_5$ is a separating $4$-cycle that separates $u_6$ and $v_2$, contradicting Lemma \ref{lem:no-sep-5-cycle}. Construct a new graph $G'$ from $G-F$ by adding the edge $u_4u_6$. Such an edge may not already exist, otherwise $u_4u_6v_4v_3v_2u_4$ is a separating $5$-cycle, contradicting Lemma \ref{lem:no-sep-5-cycle}. So, color $G'$ by $\phi$, and let $c_1$ be the color on $u_4u_6$. Apply $\phi$ to $G$. Then, apply $c_1$ to $e_{11}$, which is possible for $u_4$ saw $c_1$ on the edge $u_4u_6$ in $G'$. So, $|A_\phi(e_6)|, |A_\phi(e_7)| \ge 2$, $|A_\phi(e_8)|, |A_\phi(e_9)|, |A_\phi(e_{10})| \ge 3$, $|A_\phi(e_2)|, |A_\phi(e_4)| \ge 5$, $|A_\phi(e_3)| \ge 6$, $|A_\phi(e_5)| \ge 8$, and $|A_\phi(e_1)| \ge 9$.

\noindent\textbf{Case 1:} It is possible to apply $c_1$ to $e_9$.

Then, apply $c_1$ to $e_9$. In this case, $|A_\phi(e_6)|, |A_\phi(e_7)|, |A_\phi(e_8)| \ge 2$, $|A_\phi(e_{10})| \ge 3$, $|A_\phi(e_2)|, |A_\phi(e_4)| \ge 5$, $|A_\phi(e_3)| \ge 6$, and $|A_\phi(e_1)|, |A_\phi(e_5)| \ge 8$. We may color $|A_\phi(e_6)|, |A_\phi(e_7)|, |A_\phi(e_8)|, e_{10}$ for either $e_8, e_{10}$ do not see $e_6, e_7$ or they have an extra color. Then, $|A_\phi(e_3)|, |A_\phi(e_4)| \ge 2$, $|A_\phi(e_2)| \ge 3$, $|A_\phi(e_1)| \ge 4$, and $|A_\phi(e_5)| \ge 5$, so we may extend $\phi$ to a good coloring of $G$.

\noindent\textbf{Case 2:} It is impossible to apply $c_1$ to $e_9$.

If it is impossible to apply $c_1$ to $e_9$, then we must have $c_1 \in c(u_3)$, for $u_4$ saw $c_1$ in $G'$. Therefore, $|A_\phi(e_6)|, |A_\phi(e_7)| \ge 2$, $|A_\phi(e_8)|, |A_\phi(e_9)|, |A_\phi(e_{10})| \ge 3$, $|A_\phi(e_4)| \ge 5$, $|A_\phi(e_2)|, |A_\phi(e_3)| \ge 6$, $|A_\phi(e_5)| \ge 8$, and $|A_\phi(e_1)| \ge 9$. Color $e_6$ and $e_7$, leaving $|A_\phi(e_4)|, |A_\phi(e_8)|, |A_\phi(e_9)|, |A_\phi(e_{10})| \ge 3$, $|A_\phi(e_3)| \ge 4$, $|A_\phi(e_2)|, |A_\phi(e_5)| \ge 6$, and $|A_\phi(e_1)| \ge 7$. 
Since $e_{10}$ does not see $e_5$, we consider two cases, where $A_\phi(e_5) \cap A_\phi(e_{10}) \neq \emptyset$ and $A_\phi(e_5) \cap A_\phi(e_{10}) = \emptyset$.

\noindent\textbf{Subcase 2.1:}  $A_\phi(e_5) \cap A_\phi(e_{10}) \neq \emptyset$.

Then, give $e_5$ and $e_{10}$ the same color. Then, $|A_\phi(e_4)|, |A_\phi(e_8)|, |A_\phi(e_9)| \ge 2$, $|A_\phi(e_3)| \ge 3$, $|A_\phi(e_2)| \ge 5$, and $|A_\phi(e_1)| \ge 6$. Note that either $e_4$ does not see $e_8, e_9$ or we have undercounted its number of available colors. Thus, we may color in order $e_8, e_9, e_3, e_4, e_2, e_1$ to extend $\phi$ to a good coloring of $G$.

\noindent\textbf{Subcase 2.2:}  $A_\phi(e_5) \cap A_\phi(e_{10}) = \emptyset$.

In this case, $|A_\phi(e_4)|, |A_\phi(e_8)|, |A_\phi(e_9)|, |A_\phi(e_{10})| \ge 3$, $|A_\phi(e_3)| \ge 4$, $|A_\phi(e_2)|, |A_\phi(e_5)| \ge 6$, and $|A_\phi(e_1)| \ge 7$ with $|A_\phi(e_5) \cup A_\phi(e_{10})| \ge 9$, so we may extend $\phi$ to a good coloring of $G$.

These are all possible subcases, resolving Case 2.

 We have colored $G$ in every case, so there is no $2$-vertex on a $5$-face.
\end{proof}
	
\begin{lm}\label{5-face-two-4-vertices}
Every $5$-face contains two $4$-vertices.
\end{lm}

\begin{proof}
Let $F = v_1v_2v_3v_4v_5v_1$ be a $5$-face that has less than two $4$-vertices. Either $F$ has only $3$-vertices, or $F$ contains exactly one $4$-vertex. Label the edges as $e_1 = v_1v_2, e_2 = v_2v_3, e_3 = v_3v_4, e_4 = v_4v_5, e_5 = v_5v_1$.
                
\vspace{0.25cm}
\noindent\textbf{Case 1:} $F$ has only $3$-vertices.
We let $u_1, \dots, u_5$ be a set of vertices such that $v_i$ is adjacent to $u_i$. We never have $v_i = u_i$, $u_i = u_j$, or an edge $u_iu_j$ when $u_i$, $u_j$ would not otherwise see each other, for otherwise $G$ contains a triangle, a $4$-cycle without two $4$-vertices, or a separating $5$-cycle, contradicting either Lemma \ref{no-3-cycles}, Lemma \ref{4-cycle-two-4-vertices}, or Lemma \ref{lem:no-sep-5-cycle}. We label the additional edges of the configuration as $e_6 = v_1u_1, e_7 = v_2u_2, e_8 = v_3u_3, e_9 = v_4u_4, e_{10} = v_5u_5$. Associate to each edge $e_i$ the variable $x_i$. Set $$h(\textbf{x}) = (x_1-x_9)(x_2-x_{10})(x_3-x_6)(x_4-x_7)(x_5-x_8)(x_6-x_8)(x_6-x_9)(x_6-x_9)(x_7-x_9)(x_7-x_{10})(x_8-x_{10}).$$ Define the polynomial
            
\begin{equation*}
f(\textbf{x}) = \frac{\prod_{1 \le i < j \le 10} (x_i - x_j)}{h(\textbf{x})}.
\end{equation*}
            
The reduced polynomial is nonzero, and it contains $x_1^6x_2^6x_3^6x_4^5x_5^4x_6^3x_7x_8x_{10}^3$ with a coefficient of $-2$. Since upon coloring $G-F$ in $\phi$, we have $|A_\phi(e_i)| \ge 7$ for each $1 \le i \le 5$, and $|A_\phi(e_i)| \ge 4$ for each $6 \le i \le 10$, then we may give $G$ a good coloring by Combinatorial Nullstellens. Note that since we never have $u_i = u_j$, $u_i = v_j$, or an edge $u_iu_j$ when $u_i$, $u_j$ would otherwise see each other, this is the only polynomial we need to reduce.
             
 \vspace{0.25cm}
\noindent\textbf{Case 2:} $F$ has one $4$-vertex.
Without loss of generality we may suppose that $v_1$ is the $4$-vertex. Let $v_1$ be adjacent to $u_1, u_2$. Let $v_2, \dots, v_5$ be adjacent to $u_3, \dots, u_6$ respectively. Once again, we never have $u_i = u_j$, $u_i = v_j$, or an edge $u_iu_j$ when $u_i$, $u_j$ would not otherwise see each other, for this would violate Lemma \ref{no-3-cycles}, Lemma \ref{lem:no-sep-4-cycle}, or Lemma \ref{lem:no-sep-5-cycle}. Label the additional edges of the configuration as $e_6 = v_1u_1, e_7 = v_1u_2, e_8 = v_2u_3, e_9 = v_3u_4, e_{10} = v_4u_5, e_{11} = v_5u_6$. Associate to each $e_i$ the variable $x_i$. Then, set $h(\x) = (x_1-x_{10})(x_2-x_{11})(x_3-x_6)(x_3-x_7)(x_4-x_8)(x_5-x_9)(x_6-x_9)(x_6-x_{10})(x_7-x_9)(x_7-x_{10})(x_8-x_{10})(x_8-x_{11})(x_9-x_{11}).$ Define the polynomial 

\begin{equation*}
f(\x) = \frac{\prod_{1 \le i < j \le 11} (x_i - x_j)}{h(\textbf{x})}.
\end{equation*}
            \noindent

The reduced polynomial is nonzero, and it contains this term $x_1^5x_2^6x_3^6x_4^6x_5^5x_6^2x_7x_8^3x_9^2x_{10}^3x_{11}^3$ with a coefficient of $1$. After coloring $G - F$, we have $|A_\phi(e_i)| \ge 6$ for $1 \le i \le 2$,  $|A_\phi(e_i)| \ge 7$ for $3 \le i \le 5$, $|A_\phi(e_i)| \ge 3$ for $6 \le i \le 7$, and $|A_\phi(e_i)| \ge 4$ for $8 \le i \le 11$. Therefore, by Combinatorial Nullstellens, we may provide a good coloring for all of $G$. Note that since we never have $u_i = u_j$, $u_i = v_j$, or an edge $u_iu_j$ when $u_i$, $u_j$ would otherwise see each other, this is the only case we need to reduce.
\end{proof}

\begin{lm}\label{4-vertex-two-2-vertices-4-face}
    Let $F$ be a $4$-face with $4$-vertices $u, v$. Then, neither of $u$ or $v$ are adjacent to two $2$-vertices.
\end{lm}
    \begin{proof}
        Set $F = v_1v_2v_3v_4v_1$. By the Ore-degree condition, we may suppose that $v_1, v_3$ are $4$-vertices. By Lemma \ref{no-2-vertex-4-cycle}, there is no $2$-vertices contained in $F$. So, suppose that $v_1$ is adjacent to two $2$-vertices not on $F$. Call these vertices $u_1, u_2$, and let $v_2$ be adjacent to $u_3$, let $v_3$ be adjacent to $u_4, u_5$, and let $v_4$ be adjacent to $u_6$. Denote the edges by $e_1 = v_1v_2, e_2 = v_2v_3, e_3 = v_3v_4, e_4 = v_4v_1, e_5 = v_1u_1, e_6 = v_1u_2, e_7 = v_2u_3, e_8 = v_3u_4, e_9 = v_3u_5, e_{10} = v_4u_6$. Color $G-F$ in $\phi$, and extend $\phi$ to $G$, leaving $e_1, \dots, e_{10}$ uncolored. Then, $|A_\phi(e_8)|, |A_\phi(e_9)| \ge 3$, $|A_\phi(e_7)|,|A_\phi(e_{10})| \ge 4$, $|A_\phi(e_2)|, |A_\phi(e_3)| \ge 6$, and $|A_\phi(e_1)|, |A_\phi(e_4)|, |A_\phi(e_5)|, |A_\phi(e_6)| \ge 8$. Note that either $e_5, e_6$ do not see $e_8, e_9$, or we have undercounted their number of available colors. Therefore, after coloring $e_8, e_9$, we still have $|A_\phi(e_5)|, |A_\phi(e_6)| \ge 8$, and $|A_\phi(e_7)|,|A_\phi(e_{10})| \ge 2$, $|A_\phi(e_2)|, |A_\phi(e_3)| \ge 4$, and $|A_\phi(e_1)|, |A_\phi(e_4)| \ge 6$, so by Hall's Theorem we may extend $\phi$ to a good coloring of $G$.
    \end{proof}

\begin{lm}\label{4-vertex-two-2-vertices-5-face}
    Let $F$ be a $5$-face with two $4$-vertices $u, v$. Then, neither of $u$ or $v$ are adjacent to two $2$-vertices.
\end{lm}
    \begin{proof}
        Let $F = v_1v_2v_3v_4v_5$ be this $5$-face. By the Ore-degree condition on $G$, we may assume that $v_2, v_5$ are both $4$-vertices. Let $v_1$ be adjacent to $u_1$, let $v_2$ be adjacent to $u_2, u_3$, let $v_3$ be adjacent to $u_4$, let $v_4$ be adjacent to $u_5$, and let $v_5$ be adjacent to $u_6, u_7$. Suppose that $v_2$ is adjacent to two $2$-vertices. By Lemma \ref{no-2-vertex-5-face}, there is no $2$-vertex on $F$, so $u_2, u_3$ must be these $2$-vertices. Then, $|A_\phi(e_{11})|, |A_\phi(e_{12})| \ge 3$, $|A_\phi(e_6)|, |A_\phi(e_9)|, |A_\phi(e_{10})| \ge 4$, $|A_\phi(e_4)|, |A_\phi(e_5)| \ge 6$, $|A_\phi(e_3)| \ge 7$, $|A_\phi(e_1)|, |A_\phi(e_2)|, |A_\phi(e_7)|, |A_\phi(e_8)| \ge 8$. By Hall's Theorem, we may color $e_6, e_{10}, e_{11}, e_{12}$. Then, $|A_\phi(e_4)|, |A_\phi(e_5)| \ge 2$, $|A_\phi(e_9)| \ge 3$, $|A_\phi(e_3)| \ge 4$, $|A_\phi(e_1)| \ge 5$, $|A_\phi(e_2)| \ge 6$, and $|A_\phi(e_7)|, |A_\phi(e_8)| \ge 8$, so by Hall's Theorem, we may extend $\phi$ to color all of $G$.
    \end{proof}

\begin{lm}\label{no-4-vertex-in-two-4-faces-one-5-face}
There is no $4$-vertex contained in two $4$-faces and a $5$-face.
\end{lm}
\begin{proof}
Let $F$ be a $5$-face $v_1v_2v_3v_4v_5v_1$. By Lemma \ref{5-face-two-4-vertices}, $F$ must contain two $4$-vertices and, by Lemma \ref{no-2-vertex-5-face}, all vertices on $F$ are $3^+$-vertices, so we arbitrarily set $v_1, v_3$ to be $4$-vertices, with the remaining vertices $3$-vertices. Suppose as well that $v_1$ is contained in two $4$-faces. By Lemma \ref{no-adjacent-4-faces}, these $4$-faces do not share an edge, so one $4$-face must contain $v_2$, and the other must contain $v_5$. So, set $v_1v_2w_1w_2v_1$ and $v_1v_5u_1u_2v_1$ to be the two $4$-faces. By Lemmas \ref{4-cycle-two-4-vertices} and \ref{no-2-vertex-4-cycle}, we must have $w_1, u_1$ as $4$-vertices and $w_2, u_2$ as $3$-vertices. Label the edges as $e_1 = v_1v_2, e_2 = v_2v_3, e_3 = v_3v_4, e_4 = v_4v_5, e_5 = v_5v_1, e_6 = v_1u_2, e_7 = v_1w_2, e_8 = v_5u_1, e_9 = v_2w_1, e_{10} = u_1u_2, e_{11} = w_1w_2$, and finally set $e_{12}$ to be the edge incident to $u_2$ but not $u_1$ or $v_1$, set $e_{13}$ to be the edge incident to $w_2$ but not $w_1, v_1$, and set $e_{14}$ to be the edge incident to $v_4$ but not $v_3, v_5$ (see Figure~\ref{fig315}).

\begin{figure}[h]
\begin{tikzpicture}[scale=1.5, every node/.style={font=\small}]

\tikzstyle{vertex}=[circle, draw, inner sep=1pt, minimum size=6pt]
\tikzstyle{4vertex}=[diamond, draw, inner sep=1pt, minimum size=6pt]

\node[4vertex] (v1) at (3,2) {$v_1$};
\node[vertex] (v2) at (3,3) {$v_2$};
\node[4vertex] (v3) at (2.5,3.5) {$v_3$};
\node[vertex] (v4) at (2,3) {$v_4$};
\node[vertex] (v5) at (2,2) {$v_5$};
\node[4vertex] (u1) at (2,1) {$u_1$};
\node[vertex] (u2) at (3,1) {$u_2$};
\node[4vertex] (w1) at (4,3) {$w_1$};
\node[vertex] (w2) at (4,2) {$w_2$};

\node[vertex] (x1) at (1,3) {};
\node[vertex] (x3) at (5,2) {};
\node[vertex] (x4) at (4,1) {};

\draw (v1) -- (v2)  node[midway, right]{$e_1$};
\draw (v2) -- (v3) node[midway, right] {$e_2$};
\draw (v3) -- (v4) node[midway, left] {$e_3$};
\draw (v4) -- (v5) node[midway, right] {$e_4$};
\draw (v5) -- (v1) node[midway, above] {$e_5$};
\draw (v1) -- (u2) node[midway, right] {$e_6$};
\draw (v1) -- (w2) node[midway, above] {$e_7$};
\draw (v5) -- (u1) node[midway, right] {$e_8$};
\draw (v2) -- (w1) node[midway, above] {$e_9$};
\draw (u1) -- (u2) node[midway, above] {$e_{10}$};
\draw (w1) -- (w2) node[midway, right] {$e_{11}$};
\draw (u2) -- (x4) node[midway, above] {$e_{12}$};
\draw (w2) -- (x3) node[midway, above] {$e_{13}$};
\draw (x1) -- (v4) node[midway, above] {$e_{14}$};
\end{tikzpicture}
\caption{Figure for Lemma~\ref{no-4-vertex-in-two-4-faces-one-5-face}}
\label{fig315}
\end{figure}
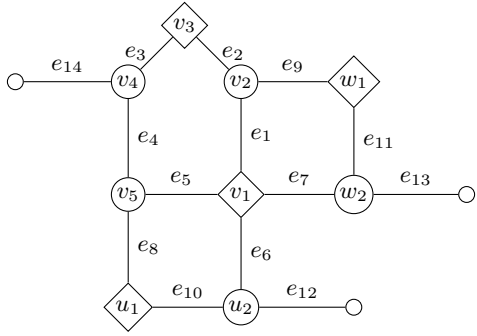

Note that there do not exist edges $v_3u_1$ and $w_1u_1$ due to the Ore-degree condition on $G$. So, construct a new graph $G'$ from $G$ by deleting $v_1, \dots, v_5, u_2, w_2$ from $G$ and adding the edges $u_1w_1$ and $v_3u_1$. We may color $G'$ in $\phi$ with $u_1w_1$ colored in $c_1$ and $v_3u_1$ colored in $c_2$. Let $\{c_3, c_4\} = c(u_1)$, and note that $c(u_1) \cap c(w_1) = \emptyset$. Apply $\phi$ to $G$, leaving the edges $e_1, \dots, e_{14}$ uncolored. We may color $e_9$ in $c_1$ and $e_2$ in $c_2$, for $v_3$ saw $c_2$ in $G'$ and $w_1$ saw $c_1$ in $G'$. We try to color $e_8, e_{10}, e_{12}$ by $c_1, c_2$. Either, we can color $e_8$ and $e_{10}$ by $c_1, c_2$, or $e_{12}$ has both $c_1, c_2$ available to it, in which case we color $e_8$ in one of $c_1, c_2$ and $e_{12}$ in one of $c_1, c_2$.

        So, apply $c_1, c_2$ to $e_8$ and one of $e_{10}$ or $e_{12}$. If $e_{12}$ is colored, then $|A_\phi(e_{13})|, |A_\phi(e_{14})| \ge 1$, $|A_\phi(e_3)|, |A_\phi(e_{10})|, |A_\phi(e_{11})| \ge 2$, $|A_\phi(e_4)| \ge 5$, $|A_\phi(e_6)|, |A_\phi(e_7)| \ge 6$, $|A_\phi(e_1)| \ge 7$, and $|A_\phi(e_5)| \ge 9$, and if $e_{10}$ is colored, then $|A_\phi(e_{12})|, |A_\phi(e_{13})|, |A_\phi(e_{14})| \ge 1$, $|A_\phi(e_3)|, |A_\phi(e_{11})| \ge 2$, $|A_\phi(e_4)| \ge 5$, $|A_\phi(e_7|, |A_\phi(e_7)| \ge 6$, $|A_\phi(e_1)| \ge 7$, and $|A_\phi(e_5)| \ge 9$. In the first case, we may color the following edge sets in order $e_{13}, e_{11}$, then $e_{12}$, then $e_{14}, e_3$, and in the second case, we may color the edge sets in order of $e_{13}, e_{11}$, then $e_{10}$, then $e_{14}, e_3$. Either, the edge sets do not see each other, in which case they may all be colored, or we have undercounted the number of colors available to them, in which case they may still be colored.

        We are left with $|A_\phi(e_6)|, |A_\phi(e_7)|, |A_\phi(e_1)| \ge 2$, and $|A_\phi(e_5)| \ge 3$. Then, if either $e_{12}$ or $e_{10}$, or one of $e_3, e_4, e_{11}, e_{13}$ are colored in one of $c_3, c_4$, we have  $|A_\phi(e_7)|, |A_\phi(e_1)| \ge 2$, $|A_\phi(e_6)| \ge 3$, and $|A_\phi(e_5)| \ge 4$, so we extend $\phi$ to a good coloring of $G$. Otherwise, we may color $e_1$ by either $c_3$ or $c_4$, leaving $|A_\phi(e_6)|, |A_\phi(e_7)| \ge 2$, and $|A_\phi(e_5)| \ge 3$, so once again we may extend $\phi$ to a good coloring of $G$.
    \end{proof}

\begin{lm}\label{no-2-vertex-adjacent-5-face-4-face}
There is no $4$-vertex contained by a $4$-face and $5$-face adjacent to a $2$-vertex.
\end{lm}

\begin{proof}
Let the $4$-face be $v_1v_2v_3v_4v_1$, and let the $5$-face be $v_2w_1w_2w_3v_3v_2$. Suppose that $v_1, v_3$ are both $4$-vertices, which is forced by Lemma \ref{4-cycle-two-4-vertices}. Then, let $v_1$ be adjacent to $u_1, u_2$, let $v_4$ be adjacent to $u_3$, and let $v_3$ be adjacent to $u_4$. Since $v_3$ is a $4$-vertex, then by our degree sum condition $w_3$ is not a $4$-vertex, so either of $w_1, w_2$ are $4$-vertices. Say that $w_3$ is adjacent to $u_5$. Let $e_1 = v_2v_3$, $e_2 = v_2v_1$, $e_3 = v_1u_1$, $e_4 = v_1u_2$, $e_5 = v_1v_4$, $e_6 = v_4u_3$, $e_7 = v_4v_3$, $e_8 = v_3u_4$, $e_9 = v_3w_3$, $e_{10} = w_3u_5$, and $e_{11} = w_3w_2$. Moreover, set $e_{16} = w_1v_2$

\noindent\textbf{Case 1:} Suppose that $w_2$ is a $4$-vertex and $w_1$ is a $3$-vertex.

Let $w_2$ be adjacent to $u_6, u_7$, and let $w_1$ be adjacent to $u_8$. Set $e_{12} = w_2u_6$, $e_{13} = w_2u_7$, $e_{14} = w_2w_1$, and $e_{15} = w_1u_8$. Note that there is no edge $u_3w_1$, nor does $u_3 = u_8$ for then we violate either Lemma \ref{lem:no-sep-4-cycle} or Lemma \ref{lem:no-sep-5-cycle}. Therefore, attach the edge $u_3w_1$ to the graph $G - \{v_1, v_2, v_3, v_4\}$, and color the resulting graph in $\phi$. Place the color $c(u_3w_1)$ on $e_6$ and $e_{16}$. Then, we have $|A_\phi(e_3)|, |A_\phi(e_4)| \ge 2$, $|A_\phi(e_9)| \ge 3$, $|A_\phi(e_5)| \ge 5$, $|A_\phi(e_2)|, |A_\phi(e_7)|, |A_\phi(e_8)| \ge 6$, and $|A_\phi(e_1)| \ge 7$. We first color in order $e_3, e_4, e_5$. Note that $e_8$ does not see $e_3, e_4$, for if an edge $u_1u_4, u_2u_4$ exists, or if $u_1 = u_4$ or $u_2 = u_4$, then we violate either Lemma \ref{lem:no-sep-4-cycle} or Lemma \ref{lem:no-sep-5-cycle}, for $u_3$ must be in the interior of the resulting $4$- or $5$-cycle. Therefore, $|A_\phi(e_8)| \ge 5$. By an identical argument with $v_3, v_3$ and $u_1, u_2$, we conclude $|A_\phi(e_9)| \ge 2$. Therefore, $|A_\phi(e_9)| \ge 2$, $|A_\phi(e_2)|, |A_\phi(e_7)| \ge 3$, $|A_\phi(e_1)| \ge 4$, and $|A_\phi(e_8)| \ge 5$. Therefore, we color the remaining set of edges by Hall's Theorem. This gives a complete coloring for the graph.

\noindent\textbf{Case 2:} Suppose that $w_2$ is a $3$-vertex and $w_1$ is a $4$-vertex.
            
Let $w_2$ be adjacent to $u_6$, and let $w_1$ be adjacent to $u_7, u_8$. Set $e_{12} = w_2u_6$, $e_{13} = w_2w_1$, $e_{14} = w_1u_7$, and $e_{15} = w_1u_8$. Following the exact same argument as above, we are able to add the edge $u_4v_1$ to the graph $G - \{v_2, v_3, w_1\}$. Color the resulting graph in $\phi$, and remove the color from $e_5$, for we will perhaps need to recolor it. Place the same color $c(u_4v_1)$ on $e_2, e_{10}$. We then have $|A_\phi(e_{14})|, |A_\phi(e_{15})| \ge 2$ and $|A_\phi(e_{13})| \ge 3$, so we color them in order. In the resulting graph, we get $|A_\phi(e_{16})| \ge 1$, $|A_\phi(e_5)| \ge 2$, $|A_\phi(e_1)|, |A_\phi(e_7)|, |A_\phi(e_9)| \ge 4$, and $|A_\phi(e_8)| \ge 6$. Color $e_{16}$ and $e_5$ in order, leaving $|A_\phi(e_1)|, |A_\phi(e_7)|, |A_\phi(e_9)| \ge 2$, and $|A_\phi(e_8)| \ge 4$. Now, we will attempt to color $e_1$ by the $3$ colors of $c(u_5) - c(u_4v_1)$. If there are not $3$ colors available, then we are good. Otherwise, if we cannot color $e_1$ in one of these three colors, then we must have at least one color $c \in c(u_5)$ in $c(w_1)$ or $c(v_4)$. In both cases, $e_9$ must see two colors twice, and thus has $3$ available colors, for a good coloring of $G$. Finally, if we can color $e_1$ in one of these three colors, then $|A_\phi(e_9)|$ remains unchanged, so we have $|A_\phi(e_7)| \ge 1$, $|A_\phi(e_9)| \ge 2$, and $|A_\phi(e_8)| \ge 3$. Therefore, once again, we get a good coloring of $G$.
 
These are all possible cases, so we have achieved a good coloring.
\end{proof}

\begin{lm}\label{no-4-vertex-in-three-5-face-4-face}
    There is no $4$-vertex in three $5$-faces and one $4$-face.
\end{lm}
    
\begin{proof}
Let $v_1$ be a $4$-vertex in three $5$-faces and one $4$-face. Let the three $5$-faces by denoted $F_1, F_2, F_3$ and denote the $4$-face by $F_4$. Let $F_2$ be the unique $5$-face which does not share an edge with $F_4$. Set $F_2 = v_1v_2v_3v_4v_5v_1$ oriented clockwise so that $v_1$ is in $F_4$. Let $F_4 = v_1w_1w_2w_3v_1$, $F_1 = w_3b_1b_2v_2v_1w_3$, and $F_3 = w_1a_1a_2v_5v_1w_1$. By symmetry, we may suppose that $v_3$ is a $4$-vertex, so that $v_4$ is a $3$-vertex. Moreover, $w_1, w_3$ are $3$-vertices and $w_2$ is a $4$-vertex, by Lemma \ref{4-cycle-two-4-vertices}. Moreover, let $u_1, u_2$ be adjacent to $v_3$ and $u_3$ be adjacent to $v_4$. We refer to this complete configuration as $H$. Now, label the edges of $H$ as follows. Set $e_1 = w_1w_2$, $e_2 = w_2w_3$, $e_3 = a_1w_1$, $e_4=w_3b_1$, $e_5 = w_1v_1$, $e_6 = w_3v_1$, $e_7 = v_1v_5$, $e_8 = v_1v_2$, $e_9 = a_2v_5$, $e_{10} = v_2b_2$, $e_{11} = v_5v_4$, $e_{12} = v_2v_3$, $e_{13} = v_4u_3$, $e_{14} = v_4v_3$, $e_{15} = v_3u_1$, and $e_{16} = v_3u_2$ (see Figure~\ref{fig317}). 

\begin{figure}[h]
\begin{tikzpicture}[scale=1.1, every node/.style={font=\small}]

\tikzstyle{vertex}=[circle, draw, inner sep=1pt, minimum size=6pt]

\node[vertex] (w1) at (4.5,4.5) {$w_1$};
\node[vertex] (w2) at (5,3) {$w_2$};
\node[vertex] (v1) at (3,4.5) {$v_1$};
\node[vertex] (w3) at (3,3.5) {$w_3$};
\node[vertex] (a1) at (5,5.5) {$a_1$};
\node[vertex] (a2) at (4,6.5) {$a_2$};
\node[vertex] (v5) at (3,5.5) {$v_5$};
\node[vertex] (v2) at (1.5,4.5) {$v_2$};

\node[vertex] (v3) at (1,5.5) {$v_3$};
\node[vertex] (v4) at (2,6.5) {$v_4$};
\node[vertex] (u3) at (2,7) {$u_3$};
\node[vertex] (u1) at (0,4.5) {$u_1$};
\node[vertex] (u2) at (0,6.5) {$u_2$};

\node[vertex] (b2) at (1,3.5) {$b_2$};
\node[vertex] (b1) at (2,2.5) {$b_1$};

\draw (w1)--(w2) node[midway,left]{$e_1$};
\draw (w2)--(w3) node[midway,below]{$e_2$};
\draw (w1)--(a1) node[midway,right]{$e_3$};
\draw (w3)--(b1) node[midway,right]{$e_4$};
\draw (w1)--(v1) node[midway,below]{$e_5$};
\draw (v1)--(w3) node[midway,left]{$e_6$};
\draw (v1)--(v5) node[midway,left]{$e_7$};
\draw (v1)--(v2) node[midway,below]{$e_8$};
\draw (a2)--(v5) node[midway,right]{$e_9$};
\draw (b2)--(v2) node[midway,left]{$e_{10}$};
\draw (v4)--(v5) node[midway,above]{$e_{11}$};
\draw (v2)--(v3) node[midway,left]{$e_{12}$};
\draw (v4)--(u3) node[midway,left]{$e_{13}$};
\draw (v3)--(v4) node[midway,above left]{$e_{14}$};
\draw (v3)--(u1) node[midway,left]{$e_{15}$};
\draw (v3)--(u2) node[midway,left]{$e_{16}$};

\draw (a1)--(a2);
\draw (b1)--(b2);

\node at (2,3.5) {$F_1$};
\node at (2,5.5) {$F_2$};
\node at (3.75,5.5) {$F_3$};
\node at (3.75,4) {$F_4$};
\end{tikzpicture}
\caption{Figure for Lemma~\ref{no-4-vertex-in-three-5-face-4-face}.}
\label{fig317}
\end{figure}

We first note that we never have any equality among the distinct vertices $a_i$, $b_i$, $v_i$, $u_i$, or $w_i$, for we are then guaranteed to have either violated the ORE conditions, violated the simplicity of $G$, or acquired a $3$-cycle, a separating $4$-cycle, or a separating $5$-cycle, violating Lemma \ref{no-3-cycles}, Lemma \ref{lem:no-sep-4-cycle}, or Lemma \ref{lem:no-sep-5-cycle}. We also never have an edge $b_i$ to $w_i$, $v_i$, or $u_i$, and likewise no edge $a_i$ to $w_i$, $v_i$, or $u_i$, for then we achieve a separating $6$-cycle, violating Lemma \ref{lem:no-sep-6-cycle}. We now use Combinatorial Nullstellensatz to reduce different possible cases. Since by Lemma \ref{5-face-two-4-vertices} every $5$-face must have two $4$-vertices, then we have four distinct cases based on which of the $a_i$, $b_i$ are $4$-vertices. Since there are no additional possible adjacency conditions between the different vertices of $H$, it is sufficient to only consider these four cases to prove this Lemma using Combinatorial Nullstellensatz.
        
If we correspond to each $e_i$ an indeterminate $x_i$, then the corresponding polynomial is the same for each case, since whether or not a particular $a_i$ or $b_i$ is a $4$-vertex has no impact on the adjacency conditions. Thus, we may reduce the same polynomial over different degree conditions. 

Let $h(\textbf{x})\cdot (x_3-x_9)\cdot (x_4-x_{10})$ be 
$$\prod_{j\in \{1,2,3,4\}}(\prod_{i=9}^{16} (x_j-x_i))\cdot \prod_{j\in \{5,6\}}(\prod_{i=13}^{16} (x_j-x_i))\cdot t(\textbf{x}),$$
where $$t(\textbf{x})=(x_3-x_4)\cdot (x_7-x_{15})(x_7-x_{16})(x_8-x_{13})(x_9-x_{10})(x_9-x_{12})(x_9-x_{15})(x_9-x_{16})(x_{10}-x_{11})(x_{10}-x_{13}).$$

This polynomial is
\[ f(\textbf{x}) = \frac{\prod_{1 \le i < j \le 16}(x_i-x_j)}{h(\textbf{x})}\]

Let $\phi$ be a good coloring of $G-H$. If for $i\in [16]$, $|A_{\phi}(e_i)|\ge a_i$, we call that {\em $H$ satisfies $(a_1, \ldots, a_{16})$}. 

\noindent\textbf{Case 1:} Both $a_1, b_1$ are $4$-vertices, and $a_2, b_2$ are $3$-vertices.

In this case, $H$ satisfies $(4,\ 4,\ 3,\ 3,\ 8,\ 8,\ 11,\ 11,\ 6,\ 6,\ 8,\ 7,\ 4,\ 6,\ 3,\ 3). $
The monomial $$x_1^3x_2^2x_3^2x_4^2x_5^7x_6^6x_7^7x_8^{10}x_9^2x_{10}^5x_{11}^7x_{12}^6x_{13}^3x_{14}^5x_{15}^2x_{16}$$ exists in $f(\textbf{x})$ with the coefficient $-1$. Therefore, this configuration is reducible by Combinatorial Nullstellensatz.
 \bigskip
 
\noindent\textbf{Case 2:} Both $a_1, b_2$ are $4$-vertices, and $a_2, b_1$ are $3$-vertices.

In this case, $H$ satisfies
$(4,\ 5,\ 3,\ 4,\ 8,\ 9,\ 11,\ 11,\ 6,\ 5,\ 8,\ 6,\ 4,\ 6,\ 3,\ 3).$
The monomial $$x_1^3x_2^4x_3^2x_4^3x_5^6x_6^8x_7^{10}x_8^6x_9x_{10}^4x_{11}^7x_{12}^5x_{13}^3x_{14}^5x_{15}^2x_{16}$$ exists in $f(\textbf{x})$ with the coefficient $2$. Therefore, this configuration is reducible by Combinatorial Nullstellensatz.
\bigskip

\noindent\textbf{Case 3:} Both $a_2, b_1$ are $4$-vertices, and $a_1, b_2$ are $3$-vertices.

In this case, $H$ satisfies $(5,\ 4,\ 4,\ 3,\ 9,\ 8,\ 10,\ 8,\ 5,\ 5,\ 7,\ 7,\ 4,\ 6,\ 3,\ 3)$.           
The monomial $$x_2^3x_3^3x_4^2x_5^8x_6^7x_7^9x_8^{10}x_9^4x_{10}x_{11}^6x_{12}^6x_{13}^3x_{14}^5x_{15}x_{16}^2$$ exists in $f(\textbf{x})$ with the coefficient $-1$. Therefore, this configuration is reducible by Combinatorial Nullstellensatz.
 \bigskip
 
\noindent\textbf{Case 4:} Both $a_2, b_2$ are $4$-vertices, and $a_1, b_1$ are $3$-vertices.

In this case, $H$ satisfies
$(5,\ 5,\ 4,\ 4,\ 9,\ 9,\ 10,\ 10,\ 5,\ 5,\ 7,\ 6,\ 4,\ 6,\ 3,\ 3). $
The monomial $$x_2^4x_3^3x_4^3x_5^8x_6^8x_7^9x_8^9x_9^4x_{11}^6x_{12}^5x_{13}^3x_{14}^5x_{15}x_{16}^2$$ exists in $f(\textbf{x})$ with the coefficient $-2$. Therefore, this configuration is reducible by Combinatorial Nullstellensatz.

We have reduced all cases.
\end{proof}

\section{discharging}

We are now ready to prove the theorem using a method of discharging.

By Euler's Formula, 
     \begin{equation*}
        \sum_{v\in V(G)}(2d(v)-6) + \sum_{f\in F(G)}(d(f)-6) = -12.
    \end{equation*}

If we assign each vertex $v$ the charge $2d(v)-6$ and each face $f$ the charge $d(f)-6$, we expect the sum of charges to be $-12$. We redistribute charges between faces and vertices to show that the sum is nonnegative, reaching a contradiction. We discharge according to the following rules.

\begin{enumerate}
    \item[R1)] Each $2$-vertex receives $1$ charge from each adjacent $4$-vertex.
    \item[R2)] Each $4$-face receives $1$ charge from each $4$-vertex in the $4$-face.
    \item[R3)] Each $5$-face receives $1/2$ a charge from each $4$-vertex in the $4$-face.
\end{enumerate}

Suppose that $G$ is a graph which is a counterexample to the claim that all planar graphs with an Ore-degree of 7 are $13$ strong edge colorable. We consider the end charges on the vertices and faces after applying rules 1-3. By Lemmas \ref{no-1-vertex}, there are no $1$-vertices, $5$-vertices, or $6$-vertices, so we do not consider them. Each $2$-vertex has a charge of $-2$, and by Lemma \ref{2-vertex-adjacent-4-vertex}, each $2$-vertex is adjacent to two $4$-vertices, so they receive $2$ charges by R1, and thus each $2$-vertex has $0$ charges. All $3$-vertices have $0$ charges. All $5^+$-vertices have nonnegative charges. Every $4$-face begins with $-2$ charges, but contains two $4$-vertices by Lemma \ref{4-cycle-two-4-vertices}, from which it receives one charge each by R2, so the $4$-face has $0$ charges. Every $5$-face contains two $4$-vertices by Lemma \ref{5-face-two-4-vertices}, and it receives $1/2$ a charge from each by R3, so it ends with $0$ charges. All $6^+$-faces have nonnegative charges.

Each $4$-vertex has $2$ charges. Now, if a $4$-vertex is in no $5^-$-faces, then it sees at most two $2$-vertices by Lemma \ref{2-vertex-adjacent-4-vertex}, so it has at least $2 - 1 - 1 = 0$ charges by R1. If a $4$-vertex is in one $4$-face, one $5$-face, or two $5$-faces then it sees at most one $2$-vertex by Lemmas \ref{4-vertex-two-2-vertices-4-face} and \ref{4-vertex-two-2-vertices-5-face}, so it has at least $0$ charges. If a $4$-vertex is in two $4$-faces, then it sees no $2$-vertices by Lemmas \ref{no-adjacent-4-faces} and \ref{no-2-vertex-4-cycle}, so it has $0$ charges. If a $4$-vertex is in a $5$-face and a $4$-face, then it sees no $2$-vertices, so it has $1/2$ a charge. If a $4$-vertex is in three $5^-$-faces, then it sees no $2$-vertices by Lemmas \ref{no-2-vertex-4-cycle} and \ref{no-2-vertex-5-face}. In this case, the $4$-vertex cannot be contained in two $4$-faces by \ref{no-4-vertex-in-two-4-faces-one-5-face}, so it has $0$ charges. If a $4$-vertex is in four $5^-$-faces, then by Lemmas \ref{no-2-vertex-adjacent-5-face-4-face} and \ref{no-4-vertex-in-two-4-faces-one-5-face} it must be in four $5$-faces. In addition, it can see no two vertices by Lemma \ref{no-2-vertex-5-face}, so it has $0$ charges.

Therefore, all vertices and faces end with a nonnegative charge, so the sum $\sum_{v\in V(G)}(2d(v)-6) + \sum_{f\in F(G)}(d(f)-6)$ is nonnegative, which contradicts Euler's Formula, so $G$ does not exist.

Thus, we have proven Theorem \ref{main-theorem}.



\begin{thebibliography}{99}

\bibitem{A99} N. Alon. Combinatorial Nullstellensatz. Combinatorics, Probability and Computing (1999) 8, 7–29.

\bibitem{A92} L. Andersen. The strong chromatic index of a cubic graph is at most 10. Discrete Math., 108(1-3):231–252, 1992. 

\bibitem{CHYZ20} L. Chen, M. Huang, G. Yu, and X. Zhou. The strong edge-coloring for graphs with small
edge weight. Discrete Math., 343(4): 111779, 2020. 


\bibitem{CHYZ19} M. Chen, J. Hu, X. Yu, S. Zhou,
List strong edge coloring of planar graphs with maximum degree 4,
Discrete Math.,
342(5): 1471-1480, 2019. 


\bibitem{GPS08} G. Greuel, G. Pfister, and H. Sch\"onemann. SINGULAR: a computer algebra system for polynomial computations. ACM Commun. Comput. Algebra 42, 3, 180–181, 2008. https://doi.org/10.1145/1504347.1504377


\bibitem{EN86} P. Erd\H{o}s and J. Ne\v{s}et\v{r}il. Irregularities of partitions, volume 8 of Algorithms and Combinatorics: Study and Research Texts, edited by G. Hal\'asz and V.T. S\'os. Springer-Verlag, Berlin, 1989. Papers from the meeting held in Fert\H{o}d, July 7–11, 1986.

\bibitem{FSGT90} R. J. Faudree, R. H. Schelp, A. Gy\'arf\'as, and Zs. Tuza. The strong chromatic index of graphs. volume 29, pages
205–211. 1990. Twelfth British Combinatorial Conference (Norwich, 1989).

\bibitem{FJ83} J.-L. Fouquet and J.-L. Jolivet. Strong edge-colorings of graphs and applications to multi-k-gons. Ars Combin., 16:141–150, 1983.

\bibitem{HS86} G. Hal\'asz and V. T. S\'os, editors. Irregularities of partitions, volume 8 of Algorithms and Combinatorics: Study
and Research Texts. Springer-Verlag, Berlin, 1989. Papers from the meeting held in Fert\"od, July 7–11, 1986.

\bibitem{H35} P. Hall. On Representatives of Subsets. J. London Math. Soc., 10(1):26–30, 1935.

\bibitem{HHT93} P. Hor\'ak, Q. He, and W. Trotter. Induced matchings in cubic graphs. J. Graph Theory, 17(2):151–160, 1993. 

\bibitem{HLZ23} M. Huang, G. Liu, and X. Zhou. Strong list-chromatic index of planar graphs with Ore-
degree at most seven. Graphs Combin., 39(6):Paper No. 116, 16, 2023. 

\bibitem{HSY18} M. Huang, M. Santana, and G. Yu. Strong chromatic index of graphs with maximum degree four.
Electron. J. Combin., 25(3):Paper No. 3.31, 24, 2018.

\bibitem{HVK22} E Hurley, R. Verclos, and R. Kang. An improved procedure for colouring graphs of
bounded local density. Adv. Comb., No. 7, 33, 2022.

\bibitem{NN15} K. Nakprasit and K. Nakprasit. The strong chromatic index of graphs with restricted Ore-degrees.
Ars Combin., 118:373–380, 2015.

\bibitem{W25} R. Wang. Strong edge-coloring of graphs with maximum edge weight seven, 2025, arXiv:2505.20345. 

\bibitem{WSWC18} Y. Wang, W. Shiu, W. Wang, M. Chen,
Planar graphs with maximum degree 4 are strongly 19-edge-colorable,
Discrete Math.,
 341 (6): 1629-1635, 2018.
\end{thebibliography}

\newpage 

\section{Appendix}

\subsection{Proof of Lemma~\ref{lem:no-sep-5-cycle}: no separating 5-cycles}

\begin{proof}
Let $C$ be a separating $5$-cycle. Once again, we may assume that $\Int(C)$ has at most $3$ interior cycles. By Lemma \ref{lem:no-adjacent-vertex-cut}, we may also assume that there is no edge cut, and thus at least two interior edges of $C$. The structure of this proof is extremely similar to the previous point.

        \vspace{0.25cm}
\noindent\textbf{Case 1:} There are no interior edges of $C$ incident to a $4$-vertex.
        \begin{proof}
First suppose that all interior edges of $C$ are incident to a $3$-vertex. Call these interior edges $e_1, e_2, e_3$. Color $\Ext(C)$ in $\psi$ using minimality. We color $\Int(C) - C$ in $\phi$. In the resulting coloring, $e_1, e_2, e_3$ each have $4$ colors available to them, one edge of $C$ has $6$ colors available to it, and every other edge has at least $9$ colors available to it. Thus, we get a set of distinct representatives using Hall's Theorem. Since $C$ is a $5$-cycle, then there are $5$ distinct colors of $C$ in the coloring of $\Ext(C)$. Therefore, we may perform a set of permutations such that the same edges of $C$ have the same colors in $\Ext(C)$ and $\Int(C)$. Glue $\Ext(C)$ and $\Int(C)$ together. Each of $e_1, e_2, e_3$ sees (possibly distinct sets of) at most $9$ colors of $\Ext(C)$. Swap $c_\phi(e_1)$ with the $10$th available color, $c_\phi(e_2)$ with the $11$th available color, and $c_\phi(e_3)$ with the $12$th available color. This gives a good coloring for all of $G$.
        \end{proof}

        \vspace{0.25cm}
        \noindent\textbf{Case 2:} There is one interior edge of $C$ incident to a $4$-vertex.
\begin{proof}
For this proof, we rely more on the structure of $C$. Set $C = v_1v_2v_3v_4v_5v_1$. Let $v_1$ be adjacent to $u_1$, let $v_2$ be adjacent to $u_2, u_3$, let $v_3$ be adjacent to $u_4$, let $v_4$ be adjacent to $u_5$, and let $v_5$ be adjacent to $u_6, u_7$. Let $e_1$ be the edge incident to the $4$-vertex in the interior of $C$. Without loss of generality, suppose this vertex is $v_2$. Let $e_2$ be the edge incident to $v_2$ in the exterior of $C$. Let $f_1 = v_1u_1, f_2 = v_3u_4, f_3 = v_4u_5$. 

We now color $\Int(C)$ and $\Ext(C)$ so that each edge of $(C \cup \{e_1, e_2, f_1, f_2, f_3\}) \cap \Int(C)$ has a distinct color, and likewise each edge of $(C \cup \{e_1, e_2, f_1, f_2, f_3\}) \cap \Ext(C)$ has a distinct color. For $\Int(C)$, we take the worst possible case, which covers all other cases. That is, suppose $f_1, f_2, f_3$ are all contained in $\Int(C)$. Color $\Int(C) - C$ by $\phi$ using minimality. We now extend this coloring. Note that $|A_\phi(f_i)| \ge 4$ for all $i \in \{1, 2, 3\}$, and $|A_\phi(e_1)| \ge 5$. Moreover, $|A_\phi(v_3v_4)| \ge 7$, $|A_\phi(v_1v_2)|, |A_\phi(v_2v_3)| \ge 8$, and $|A_\phi(v_4v_5)|, |A_\phi(v_5v_1)| \ge 10$. Therefore, by Hall's Theorem, there exists a system of distinct representatives for the edges of $C, e_1, f_1, f_2, f_3$, as desired.

We now consider the exterior of $C$. There are two worst possible cases here. By Lemma \ref{lem:no-adjacent-vertex-cut}, $f_1, f_2, f_3$ may not all be in $\Ext(C)$. Either $f_1, f_2$ must be in $\Int(C)$ or $f_3$ must be in $\Int(C)$. Thus, we color $\Ext(C) - C$ in $\psi$, assuming either that $f_3$ is in $\Ext(C)$ or assuming that $f_1, f_2$ are in $\Ext(C)$. Begin with the first case. Then, $|A_\psi(v_5u_6)|, |A_\psi(v_5u_7)| \ge 3$, $|A_\psi(f_3)| \ge 4$, $|A_\psi(e_2)| \ge 5$, $|A_\psi(v_4v_5)| \ge 6$, and $|A_\psi(v_1v_2)|, |A_\psi(v_2v_3)|, |A_\psi(v_3v_4)|, |A_\psi(v_5v_1)| \ge 9$. Again, by Hall's Theorem, we may assign a system of distinct representatives to each of the edges. Lastly, if instead $f_1, f_2$ are in $\Ext(C)$, then $|A_\psi(v_5u_6)|, |A_\psi(v_5u_7)| \ge 3$, $|A_\psi(f_1)|, |A_\psi(f_2)| \ge 4$, $|A_\psi(e_2)| \ge 5$, $|A_\psi(v_1v_4)| \ge 6$, $|A_\psi(v_1v_2)|, |A_\psi(v_2v_3)| \ge 8$, $|A_\psi(v_4v_5)| \ge 9$, and $|A_\psi(v_3v_4)| \ge 10$. Once again, we achieve a system of distinct representatives using Hall's Theorem. Every other case is a subcase of these two cases, so we may always get a system of distinct representatives.

 Now, permute the colors of $C$ in $\Int(C)$ and $\Ext(C)$, and glue $\Int(C)$ and $\Ext(C)$ back together. If $c_\phi(e_1)$ is bad, then $c_\phi(e_1)$ cannot be swapped with at most $10$ colors, consisting of $5$ colors on $C$, and an additional $5$ in $\Ext(C)$. So, swap $c_\phi(e_1)$ with the $11$th available color. We now swap $c(f_1), c(f_2)$, and $c(f_3)$, which may be in the interior or exterior of $C$. Each $c(f_i)$ cannot be swapped with the $5$ colors on $C$, the color $c_\phi(e_1)$ nor the at most $2$ edges of $\Ext(C)$ incident to a neighboring vertex of $f_i$. This is a total of $8$ colors. So, there is sufficient space to swap $c(f_1), c(f_2)$, and $c(f_3)$. Finally, in the case that $e_2$ sees a bad color on the endpoint of $e_1$, we swap $c_\psi(e_2)$. Then, $c_\psi(e_2)$ cannot be permuted with any color on $C$, nor can it be permuted with $c_\psi(e_1)$ nor the two colors on edges incident to the endpoint of $e_1$, nor the three $c(f_i)$. This is a total of $11$ colors, so we swap $c_\psi(e_2)$ with the $12$th available color. This gives a good color for all of $C$.
        \end{proof}

        \vspace{0.25cm}
        \noindent\textbf{Case 3:} There are two interior edges of $C$ incident to a $4$-vertex.

\begin{proof}
This case splits into two separate subcases, one in which the two interior edges are incident to the same vertex, and one in which they are incident to different vertices. Set $C = v_1v_2v_3v_4v_5v_1$. Let $v_1$ be adjacent to $u_1$, let $v_2$ be adjacent to $u_2, u_3$, let $v_3$ be adjacent to $u_4$, let $v_4$ be adjacent to $u_5$, and let $v_5$ be adjacent to $u_6, u_7$.

            \vspace{0.25cm}
            \noindent\textbf{Subcase 3.1:} There are two interior edges of $C$ incident to the same $4$-vertex.
\begin{proof}[Subproof]
 Without loss of generality, let $e_1, e_2$ both be adjacent to $v_2$, and suppose that $e_1, e_2$ are interior edges of $C$. We may additionally assume that there is one interior edge incident to a $3$-vertex, which must be $v_4$, for if there are two edges incident to a $3$-vertex then we consider $\Ext(C)$, and if the interior edge is incident to $v_1$ or $v_3$, then we violate Lemma \ref{lem:no-adjacent-vertex-cut}. We color the interior and exterior such that each color on $C$ is unique. So, in $\Ext(C)$, after coloring $\Ext(C) - C$ by $\psi$ using minimality, we have $|A_\psi(v_5u_6)|, |A_\psi(v_5u_7)| \ge 3$, $|A_\psi(v_1u_1)|, |A_\psi(v_3u_4)| \ge 4$, $|A_\psi(v_1v_5)| \ge 6$, and $|A_\psi(v_1v_2)|, |A_\psi(v_2v_3)|, |A_\psi(v_3v_4)|, |A_\psi(v_4v_5)| \ge 9$. Thus, using Hall's Theorem, we get a system of distinct representatives for the edges. Likewise, after coloring $\Int(C) - C$ by $\phi$, we get $|A_\phi(e_1)|, |A_\phi(e_2)| \ge 3$, $|A_\phi(v_4u_5)| \ge 4$, and $|A_\phi(v_1v_2)|, |A_\phi(v_2v_3)|, |A_\phi(v_3v_4)|, |A_\phi(v_4v_5)|$, $|A_\phi(v_5v_1)| \ge 8$. Thus, again by Hall's Theorem, we get a system of distinct representatives. Now, permute the colors of $C$ in $\Int(C)$ and $\Ext(C)$ to be identical, and glue $\Int(C)$ and $\Ext(C)$ back together. First, if $c_\phi(e_1), c_\phi(e_2)$ are bad, then $c_\phi(e_1), c_\phi(e_2)$ cannot be permuted with $7$ colors of $\Ext(C)$, including the $5$ colors on $C$. Thus, there is enough space to permute them with the remaining colors to give them a good coloring. Then, if $c_\phi(v_4u_5)$ is bad, this color cannot be permuted with at most $8$ colors of $\Ext(C)$, plus the additional two on $c_\phi(e_1), c_\phi(e_2)$. This is $10$ colors, so we permute $c_\phi(v_4u_5)$ with the $11$th available color. Thus, we have achieved a good color for $G$.
            \end{proof}

 \vspace{0.25cm}
\noindent\textbf{Subcase 3.2:} There are two interior edges of $C$ incident to the distinct $4$-vertices.
\begin{proof}[Subproof]
For identical reasons to Subcase 3.1, we need only consider there being one additional interior edge incident to a $3$-vertex. Call the edges incident to $3$-vertices $f_1, f_2, f_3$. Let $g_1, g_2, g_3, g_4$ be incident to the two $4$-vertices of $C$. Say that $g_1, g_2$ are incident to the same $4$-vertex, and $g_3, g_4$ are incident to the same $4$-vertex. Let $g_1, g_4$ be in the interior, and $g_2, g_3$ be in the exterior. Moreover, let $g_1, g_2$ be incident to $v_2$, and let $g_3, g_4$ be incident to $v_4$. Finally, let the additional endpoint of $g_i$ be $u_i$.

So, first color $\Ext(C) - C$ in $\psi$. In the worst case, all the $f_i$ are exterior edges. So, we then have $|A_\psi(f_i)| \ge 4$ for each $1 \le i \le 3$, and $|A_\psi(g_2)|, |A_\psi(g_3)| \ge 5$. All edges of $C$ have $8$ colors available to them, except $v_3v_4$, which has $7$ colors available to it. If any of the $3$ vertices are interior, then two edges of $C$ have $3$ more colors available to them, and this is sufficient to color all edges distinctly with Hall's Theorem. Otherwise, we color in order $f_1, f_2, f_3, g_2, g_3$. We also permit $v_1v_2$ and $v_5v_1$ to be colored identically to the $f_i$ they do not see. We can guarantee they do not see the $f_i$ furthest opposite them, for otherwise there must exist a $3$-cycle or separating $4$-cycle, violating Lemma \ref{no-3-cycles} or Lemma \ref{lem:no-sep-4-cycle}. Then, $|A_\phi(v_3v_4)| \ge 2$, $|A_\phi(v_2v_3)|, |A_\phi(v_4v_5)| \ge 3$, and $|A_\phi(v_1v_2)|, |A_\phi(v_6v_1)| \ge 4$. If $A_\psi(v_2v_3) \cap A_\psi(v_4v_5) \neq \emptyset$, then we color them the same. So, suppose that we may give them the same color. Then, we color in order $v_3v_4, v_1v_2, v_1v_6$ and achieve a good coloring for $\Ext(C)$. Otherwise, we can not color them identically, in which case $A_\psi(v_2v_3) \cap A_\psi(v_4v_5) = \emptyset$, forcing $|A_\psi(v_2v_3) \cup A_\psi(v_4v_5)| \ge 6$. Therefore, we color in order $v_3v_4, v_1v_2, v_6v_1, v_2v_3, v_4v_5$ to achieve a good coloring for $\Ext(C)$. We now color $\Int(C) - C$ in $\phi$. Now, there is at most one internal $f_i$, so we can guarantee that there are two $2$-vertices in $\Int(C)$ on the cycle $C$. Therefore, there are $4$ cycle edges with at least $10$ available colors, and the remaining $2$ have at least $8$ cycle colors. The internal $f_i$ has $3$ available colors, and $|A_\phi(g_1)|, |A_\phi(g_4)| \ge 4$. Therefore, if $c_\psi(v_2v_3) = c_\psi(v_4v_5)$, then since $|A_\phi(v_2v_3) \cap A_\phi(v_4v_5)| \ge 3$, given that $|A_\phi(v_2v_3)|, |A_\phi(v_4v_5)| \ge 8$, we may first color $v_2v_3, v_4v_5$ identically, and then the remainder of the edges uniquely by Hall's Theorem. Otherwise, if $c_\psi(v_2v_3) \neq c_\psi(v_4v_5)$, then we just color all edges distinctly using Hall's Theorem.

We now join $\Int(C)$ and $\Ext(C)$, after permuting $C$ in $\Int(C)$ so that $C$ is colored identically in $\phi$ and $\psi$. We first permute the exterior of $C$. So, $c_\psi(g_2), c_\psi(g_3)$ must avoid the $5$ colors of $C$, plus the $3$ colors $c_\phi(u_1), c_\phi(u_4)$ respectively, and the at most $1$ color from the interior $f_i$, for a total of $9$. We permute then with the $10$th and $11$th available. We then permute $c_\phi(g_1), c_\phi(g_4)$. Each of these must avoid the $5$ colors of $C$, the $3$ colors of $c_\psi(u_2), c_\psi(u_5)$ respectively, and the $3$ colors of the $f_i$, for a total of $11$. We permute with the $12$th and $13$th. Finally, we must permute the one interior $f_i$. This must avoid the at most $4$ additional edges of $\Ext(C)$, the $5$ cycle colors, and the $2$ colors $c_\phi(g_1), c_\phi(g_4)$. We otherwise permit this edge to permute with $c_\phi(u_1) - c_\phi(g_1), c_\phi(u_4) - c_\phi(g_4)$, for we earlier permuted $c_\psi(g_2), c_\psi(g_3$ to be good with the interior $f_i$. Therefore, $f_\phi(f_i)$ must avoid at most $11$ colors, so we may permute it with the $12$th available. We have resolved all color conflicts, and thus achieved a good coloring for $G$.
            \end{proof}
            In every case, we have a good coloring, so there is no separating $5$-cycle $C$ with two interior edges incident to a $4$-vertex.
        \end{proof}

We need not consider any more cases, for if there are three interior edges of a $5$-cycle incident to a $4$-vertex, then there is only one exterior edge incident to a $4$-vertex, which is symmetric to Case 2. Likewise if there are four interior edges incident to a $4$-vertex.
\end{proof}

\subsection{Proof of Lemma~\ref{lem:no-sep-6-cycle}: no separating 6-cycles}

    \begin{proof}
        Let $C$ be such a $6$-cycle. We first claim that it is sufficient to suppose that our $6$-cycle has three $4$-vertices and three $3$-vertices, so long as our proof satisfies some conditions. Suppose first that the separating $6$-cycle contains a $2$-vertex, $u$. Then, suppose that the graph is reducible if $u$ is replaced by a $3$-vertex with an additional edge $e$ is reducible. We may remove the edge $e$ to get a coloring when $u$ is a $2$-vertex. So, we may assume that there are no $2$-vertices in $C$. Now, consider a particular configuration $C$ that has a $3$-vertex with edge $e$ which may be swapped with a $4$-vertex $v$ without violating the ORE condition. Call this swapped cycle $C'$, and say that $vu_1, vu_2$ are two edges incident to $v$ which are not embedded in $C'$. We may additionally assume that $vu_1, vu_2$ are exterior to $C'$ if $e$ is exterior to $C$, and interior to $C'$ if $e$ is interior to $C$. Suppose we prove that $C'$ is a reducible configuration. Suppose further that we always assume that $|c(u_1) \cup c(u_2)$ is maximal. Thus, the configuration $C'$ must be colored by strictly more colors than $C$, and every edge incident to $v$ has strictly fewer colors available to it in $C'$ than in $C$. Therefore, a coloring of $C'$ under these conditions is a strictly worse case, and so if $C'$ is colorable, then $C$ is also colorable.

        Set $C = v_1v_2v_3v_4v_5v_6v_1$. Let $e_i = v_iv_{i+1}$ for $i \le 5$, and let $e_6 = v_6v_1$. Suppose that $v_1, v_3, v_5$ are $3^-$-vertices and $v_2, v_4, v_6$ are $4$-vertices. Without loss of generality, suppose that $f_1, f_2, f_3$ are incident to $3$-vertices $v_1, v_3, v_5$. Likewise, let the pair $g_1, g_2$, the pair $g_3, g_4$, and the pair $g_5, g_6$ be adjacent to $v_2, v_4, v_6$ respectively. Say also that $g_i$ has an additional endpoint $u_i$.

        \vspace{0.25cm}
        \noindent\textbf{Case 1:} Suppose that none of the $g_i$ are interior edges.
            \begin{proof}
                Color $\Ext(C)$ by minimality, and note that at most $f_1, f_2, f_3$ are interior edges. Color $\Int(C) - C$ in $\phi$. We claim that we may extend $\phi$ to $C$ such that if $c(e_i) = c(e_j)$ in $\Ext(C)$, then $c(e_i) = c(e_j)$ in $\Int(C)$. So, $|A_\phi(e_i)| \ge 10$ for each edge $e_i$ on $C$. Thus, for $i \le 3$, we have $|A_\phi(e_i) \cap A_\phi(e_i)| \ge 7$. Therefore, by Hall's Theorem, we may color all three pairs of edges on $C$ identically. After each $e_i, e_j$ satisfying $c(e_i) = c(e_j)$ in $\Ext(C)$ is identically colored, we may color the remaining edges, which satisfy $|A_\phi(e_k)| \ge 8$, and thus may be colored distinctly by Hall's Theorem. Note then that $|A_\phi(f_i)| \ge 1$. None of the $f_i$ see each other through $C$, since $C$ has three $4$-vertices, so we may color each $f_i$. Now, we permute $\Int(C)$, so that we may join $\Ext(C)$ and $\Int(C)$ back together. We now recolor the $c(f_i)$ in the event that their colors are bad in the new coloring of $G$. If $c(f_i) = c(f_j) = c(f_k)$ or $c(f_i) = c(f_j)$, then there are at most $12$ colors we cannot swap $c(f_i)$ with, consisting of the at most $6$ colors of $C$, and the at most $6$ colors $c(g_i)$. Thus, we swap with the $13$th available color. For the remaining distinctly colored $c(f_j)$, each $c(f_j)$ cannot be swapped with at most $10$ colors, consisting of the at most $6$ colors on $C$  plus the $4$ colors on additional edges in $\Ext(C)$. Thus, there are at least $3$ colors we may swap each $c(f_j)$ with. Since there are at most $3$ edges $c(f_j)$, then we may permute the $c(f_j)$ in $\Int(C)$ to get a good coloring of $G$. 
            \end{proof}

        \vspace{0.25cm}
        \noindent\textbf{Case 2:} Suppose that there is one interior edge of $C$ incident to a $4$-vertex.
            \begin{proof}
		    Suppose, without loss of generality, that $g_1$ is the interior edge. We must color $\Ext(C)$ so that it has at most one repeat edge on $C$. Now, not all the $f_i$ may be exterior edges, for then $g_1$ is a cut edge, contradicting Lemma \ref{lem:no-adjacent-vertex-cut}. Therefore, we may guarantee that when coloring $\Ext(C)$, there is at least one $2$-vertex $u$ in $\Ext(C)$. Suppose first that there are two external edges. Then, we must have $u = v_5$, for if $u \in \{v_1, v_3\}$, then $\{u, v_2\}$ is an edge cut, contradicting Lemma \ref{lem:no-adjacent-vertex-cut}. We begin with this case.

		\vspace{0.25cm}
                \noindent\textbf{Subcase 2.1:} Both $f_1, f_2$ are exterior edges and $f_3$ is an interior edge. 
			\begin{proof}[Proof of Subcase]
				Since $f_3$ is interior, we may make a new graph $H$ by attaching the edge $v_2v_5$ to $\Ext(C)$ without violating planarity or the degree sum condition. Color $H$ by minimality, and place the resulting coloring on $\Ext(C)$. Note that $g_2$ is now colored distinctly from $e_4, e_5$ because of the existence of $v_2v_5$. Thus, $g_2$ is colored distinctly from all colors on $C$. Then, color $\Int(C) - C$ in $\phi$. Place $\phi$ on $\Int(C)$. We will extend it to all of $\Int(C)$. First, $|A_\phi(f_3)| \ge 4$, $|A_\phi(g_1)| \ge 5$, and $|A_\phi(e_i)| \ge 10$, for all $i$. For every pair of edges $e_i, e_j$, if these edges are colored identically in $C$, then we have $|A_\phi(e_i) \cap A_\phi(e_j)| \ge 7$. Therefore, we are able to first color identically the pairs of edges on $C$ which must be colored identically, and then color the remainder by Hall's Theorem. We can further guarantee that for each color is distinct.

				Now, permute $\Int(C)$ and $\Ext(C)$ such that $C$ is colored identically in both graphs. Then, join $\Int(C)$ and $\Ext(C)$ together. We first permute $g_2$ in $\Ext(C)$. This must avoid the at most $6$ colors of $C$, the $3$ colors of $c(u_1)$, and the $2$ colors $c(f_1), c(f_2)$. So, we permute with the $12$th available color. We note that when permuting $c(g_2)$, we may also permute a color on some other $c(g_i)$ or $c(f_3)$, but we do not permute any colors of $C$. We then permute $c(g_1)$. This must avoid the $6$ colors of $C$ and the $3$ colors of $c(u_2)$, so we permute it with the $10$th available color. We then permute $c(f_1)$ and $c(f_2)$. Both must avoid the at most $6$ colors of $C$, the $3$ colors of neighboring $c(g_i)$, and the color $c(g_1)$. We note that we may permute with $c(u_1) - c(g_1)$, for either $c(f_i)$ was swapped with $c(g_1)$ in the previous step, and is thus good with $c(g_2)$, or $c(f_i)$ was not swapped, in which case we permuted $c(g_2)$ to be good with $c(f_i)$. In either case, permuting with $c(u_1) - c(g_1)$ therefore introduces no new color conlicts. Thus, each $c(f_i)$ has $10$ colors it must avoid, so we permute with the $11$th and $12$th for a good coloring of $G$. 
			\end{proof}

		\vspace{0.25cm}
		\noindent\textbf{Subcase 2.2:} There are at least two interior edges $f_i$.
			\begin{proof}[Proof of Subcase]
				We color $\Ext(C) - \{v_1, v_2, v_3\}$ in $\psi$. There are two worst cases we consider. First, it may be that $f_1, f_2$ are both interior edges. In this case, we have $|A_\psi(e_6)|, |A_\psi(e_3)| \ge 4$, $|A_\psi(g_2)| \ge 5$, and $|A_\psi(e_1)|, |A_\psi(e_2)| \ge 8$. We want to guarantee that $e_1, e_2, g_2$ are colored distinctly from $e_4, e_5$, so we remove at most $2$ colors from their list of available colors, giving $|A_\psi(g_2)| \ge 3$, $|A_\psi(e_6)|, |A_\psi(e_3)| \ge 4$, and $|A_\psi(e_1)|, |A_\psi(e_2)| \ge 6$. We then color the edges distinctly by Hall's Theorem. In the other worst case, we have either $f_1, f_3$ or $f_2, f_3$ as interior edges. The case is symmetric, so suppose that $f_2, f_3$ are interior edges. Then, $|A_\psi(f_1)| \ge 1$, $|A_\psi(e_6)| \ge 2$, $|A_\psi(g_2)|, |A_\psi(e_1)|, |A_\psi(e_3)| \ge 5$, and $|A_\psi(e_2)| \ge 8$. For $e_1, e_2, g_2$, we want to guarantee these are not colored the same as one of $e_4, e_5$, so we are left with $|A_\psi(f_1)| \ge 1$, $|A_\psi(e_6)| \ge 2$, $|A_\psi(g_2)| \ge 3$, $|A_\psi(e_1)| \ge 4$, $|A_\psi(e_3)| \ge 5$, and $|A_\psi(e_2)| \ge 7$. Thus, we are able to color the remaining edges distinctly by Hall's Theorem. We now color $\Int(C)$. So, color $\Int(C) - C$ in $\phi$. Then, place $\phi$ on $\Int(C)$. Suppose all the $f_i$ are interior, for this is the worst case. So, $|A_\phi(f_i)| \ge 3$ for all $1 \le i \le 3$, $|A_\phi(g_1)| \ge 5$, $|A_\phi(e_1)|, |A_\phi(e_2)| \ge 8$, and $|A_\phi(e_3)|, |A_\phi(e_4)|, |A_\phi(e_5)|, |A_\phi(e_6)| \ge 10$. By Hall's Theorem, we may color all such edges distinctly.

			We now permute $\Int(C)$ and $\Ext(C)$ such that $C$ is colored identically in both graphs. Join $\Int(C)$ and $\Ext(C)$ together. We begin by permuting the exterior, then the interior. First, we permute $c(g_2)$. This must avoid the $6$ colors of $C$, the $3$ colors of $c(u_1)$, and the $3$ colors $c(f_1), c(f_2), c(f_3)$, for a total of $12$. We permute with the $13$th available. We then permute $c(g_1)$. This must avoid the $6$ colors of $C$, the $3$ colors of $c(u_2)$, and the at most $1$ adjacent color $c(f_i)$. Finally, we permute the $c(f_i)$. There are at most $3$. Each must avoid the $6$ colors of $C$ and the color $c(g_1)$. As in Subcase 2.1, we may permute with $c(u_1) - c(g_1)$. Two $c(f_i)$ must avoid $3$ additional colors in $\Ext(C)$, and one must avoid $4$. This makes for two that must avoid $10$ colors, and one that must avoid $11$, which is $f_3$. Therefore, we permute $c(f_3)$ with its $12$th color, $c(f_1)$ with its $12$th color, and $c(f_2)$ with its $13$th color. Thus, we have achieved a good coloring for all of $G$.
			\end{proof}
		We have covered every case.
	\end{proof}

        \vspace{0.25cm}
        \noindent\textbf{Case 3:} Suppose that there are two interior edges of $C$ incident to a $4$-vertex.
            \begin{proof}
		We will need to consider two separate cases.

                \vspace{0.25cm}
                \noindent\textbf{Subcase 3.1:} Suppose that the two interior $g_i$ are adjacent to different $4$-vertices.
                	\begin{proof}[Subproof]
				Without loss of generality, we may assume that $g_1, g_6$ are both interior edges. We must consider several different configurations.

				\vspace{0.25cm}
				\noindent\textbf{Config. 3.1.1:} We first suppose that all the $f_i$ are exterior edges.

					Color $\Ext(C)$ by minimality. We now color $\Int(C) - C$ in $\phi$. We will extend $\phi$ to $C$, but in doing so, we must guarantee that the colors of $C$ have the same configuration in $\Int(C)$ as in $\Ext(C)$. Moreover, it may be that $c_\psi(g_2), c_\psi(g_5)$ are colored identically to $c(e_4), c(e_5)$ or $c(e_2), c(e_3)$ respectively. Therefore, we must guarantee that, after coloring $\Int(C)$, we have $c(e_4), c(e_5) \not \in c_\phi(g_1)$, and likewise that $c(e_2), c(e_3) \not \in c_\phi(g_6)$. If these are satisfied, then whenever $c_\psi(g_2), c_\psi(g_5)$ are on colors of $C$, then they are automatically good with the colors incident to the endpoints $u_1, u_6$ of $g_1, g_6$. So, we forbid $c_\phi(u_1)$ from $A_\phi(e_5)$, $A_\phi(e_4)$, and we forbid $c_\phi(u_6)$ from $A_\phi(e_2), A_\phi(e_3)$. Thus, we have $|A_\phi(g_1)|, |A_\phi(g_6)| \ge 5$, $|A_\phi(e_2)|, |A_\phi(e_5)| \ge 9$, and $|A_\phi(e_1)|, |A_\phi(e_3)|, |A_\phi(e_4)|, |A_\phi(e_6)| \ge 11$. Note that for any pair $e_i, e_j$, we have $|A_\phi(e_i) \cap A_\phi(e_j)| \ge 5$. Therefore, we may color all pairs of edges $e_i, e_j$ identically or uniquely, and obtain a good coloring of $\Int(C)$ in which $g_1, g_6$ are colored uniquely from all colors of $C$.

					Now, permute $C$ in $\Int(C)$ so that we may join $\Int(C)$ and $\Ext(C)$ back together. We first begin by permuting $\Ext(C)$. If either of $c_\psi(g_2), c_\psi(g_5)$ are on $C$, then we do not permute them. Otherwise, first assuming they are distinct, each must avoid the at most $6$ colors of $C$ and the $3$ colors of $c_\phi(u_1), c_\phi(u_6)$ respectively, for a total of $9$. Therefore, we permute them with the $10$th and $11$th available colors. If, alternatively, $c_\psi(g_2) = c_\psi(g_5)$, then this color must avoid the at most $6$ colors of $C$, and the at most $6$ colors of $c_\phi(u_1) \cup c_\phi(u_6)$, for a total of $12$. We permute with the $13$th available color. We now permute $\Int(C)$. Recall that we have colored $g_1, g_6$ uniquely, so $c_\phi(g_1) \neq c_\phi(g_6)$. Now, both must avoid the $6$ colors of $C$, the $3$ colors of $c_\psi(u_2)$, $c_\psi(u_5)$ respectively, and the $2$ colors from adjacent $f_i$, for a total of $11$. We permute with then with the $12$th and $13$th available. Therefore, we have achieved a good coloring for this configuration.

				\vspace{0.25cm}
				\noindent\textbf{Config. 3.1.2:} Now suppose $f_1$ is an interior edge.

					We will need to perform slightly different colorings of $\Ext(C)$ depending on whether $i=1$ or $i = 2, 3$. Note that the case of $i=2$ and $i=3$ are symmetric, so we need only consider two separate colorings. First, suppose that $f_1$ is an interior edge. Color $\Ext(C) - \{v_1, v_2, v_6\}$ in $\psi$. We require that $g_2, g_5, e_1, e_6$ is colored distinctly from all colors on $C$. Therefore, $|A_\psi(g_2)|, |A_\psi(g_5)| \ge 2$, $|A_\psi(e_2)|, |A_\psi(e_5)| \ge.$, and $|A_\psi(e_1)|, |A_\psi(e_6)| \ge 7$. We note that $e_2$ cannot see $g_5, e_5$, and likewise $e_5$ cannot see $e_2$, $g_2$. We require $g_2, g_5$ to be colored distinctly, but $e_2, e_5$ may be colored identically. We enforce that $e_1, e_6$ are colored distinctly from all other colors on $C$ and $g_2, g_5$, so that only $c(e_2) = c(e_5)$ may hold among the edges of $C$. We color in order $g_2, g_5, e_2, e_5, e_1, e_6$, providing a good coloring in which only $c(e_2) = c(e_5)$ may hold on the edges of the cycle. We then color $\Int(C) - C$ in $\phi$, and place $\phi$ on $C$. We have $|A_\phi(f_1)| \ge 4$, $|A_\phi(g_1)|, |A_\phi(g_6)| \ge 5$, $|A_\phi(e_1)|, |A_\phi(e_6)| \ge 8$, $|A_\phi(e_2)|, |A_\phi(e_5)| \ge 11$, and $|A_\phi(e_3)|, |A_\phi(e_4)| \ge 13$. We must have $|A_\phi(e_2) \cap A_\phi(e_5)| \ge 9$, so if $c_\psi(e_2) = c_\psi(e_5)$, then we give $e_2, e_5$ the same color, and then color in order $f_1, g_1, g_6, e_1, e_6, e_3, e_4$. Otherwise, we color in order $f_1, g_1, g_6, e_1, e_6, e_2, e_5, e_3, e_4$.

				Now, permute $C$ in $\Int(C)$ to have the same color as $C$ in $\Ext(C)$. We begin with $g_2, g_5$. Note that these may have the same color as $g_3, g_4$. So, we $c_\psi(g_2)$ and $c_\psi(g_5)$ must avoid the at most $6$ colors of $C$, the $3$ colors of $c_\phi(u_1), c_\phi(u_6)$ respectively, and the $1$ color $c_\phi(f_1)$. We also require that both avoid $c_\phi(g_1), c_\phi(g_6)$. This is a total of $11$ colors, so we permute with the $12$th and $13$th respectively (recall we colored $\Ext(C)$ such that $c_\psi(g_2) \neq c_\psi(g_5)$). We then permute $c_\phi(g_1), c_\phi(g_6)$. These must avoid the $6$ colors of $C$, the $3$ colors of $c_\psi(u_2)$, and $1$ color of $c_\psi(f_2)$, so we swap with the $12$th and $13$th respectively. Note that neither must avoid $c_\phi(u_1) -c_\phi(g_1)$ or $c_\phi(u_6) - c_\phi(g_6)$, for we required that both $c_\phi(g_1), c\phi(g_6)$ be good with $c_\psi(g_2), c_\psi(g_5)$. Moreover, $f_1$ is now good, since $c_\psi(g_2), c_\psi(g_5)$ were colored not to conflict with $c_\phi(f_1)$. There are now no conflicts between $\phi$ and $\psi$, so we have a good coloring of $G$.

				\vspace{0.25cm}
				\noindent\textbf{Config. 3.1.3:} Now suppose that either $f_2$ or $f_3$ is an interior edge.

					So, suppose that $f_2$ is an interior edge, and $f_1, f_3$ exterior. Color $\Ext(C) - C$ in $\psi$. We then have $|A_\psi(g_3)|, A_\psi(g_4)| \ge 3$, $|A_\psi(f_1)|, |A_\psi(f_3)| \ge 4$, $|A_\psi(g_2)|, |A_\psi(g_5)| \ge 5$, $|A_\psi(e_4)| \ge 6$, $|A_\psi(e_1)|, |A_\psi(e_5)|, |A_\psi(e_6)| \ge 8$, $|A_\psi(e_3)| \ge 9$, and $|A_\psi(e_2)| \ge 11$. We first color $g_2, g_5$ distinctly. We then color in order $g_3, g_4, f_1, f_3$. When performing this coloring, we permit $c_\psi(g_2), c_\psi(g_5)$ to equal $c_\psi(g_3), c_\psi(g_4)$, and we permit $c_\psi(g_2) = c_\psi(f_3)$. We also allow $c_\psi(f_1)$ to equal $c_\psi(f_3), c_\psi(g_3), c_\psi(g_4)$. Thus, upon coloring $g_3, g_4$, we have $|A_\psi(g_3)|, |A_\psi(g_3)| \ge 3$, $|A_\psi(f_3)| \ge 3$, and $|A_\psi(f_1)| \ge 2$. Thus, everything may be colored. We then color $C$. We require that none of $C$ be colored identically to $c_\psi(c_2), c_\psi(g_5)$. Otherwise, we permit $c_\psi(e_i)$ to equal $c_\psi(g_3), c_\psi(g_4)$, or $c_\psi(f_i)$. Thus, we have $|A_\psi(e_4)| \ge 1$, $|A_\psi(e_5)| \ge 2$, $A_\psi(e_6)| \ge 4$, $|A_\psi(e_1)| \ge 5$, and $|A_\psi(e_2)| \ge 6$. Therefore, we color the remaining set of edges distinctly using Hall's Theorem. Now, we have all colors on $C$ distinct. We may not have $c_\psi(g_2), c_\psi(g_5)$ as a color of $C$. We may have $c_\psi(g_2), c_\psi(g_5)$ as a color of $c_\psi(g_3), c_\psi(g_4), c_\psi(f_1), c_\psi(f_3)$. We now color $\Int(C)$. So, color $\Int(C) - C$ in $\phi$. We may color this case identically to the coloring of $\Int(C)$ in Config. 3.1.2, in the case when there are no repeat colorings.

					Now, we first permute $c_\psi(g_2), c_\psi(g_5)$. These both must avoid the $6$ colors of $C$, plus the $3$ colors of $c_\phi(u_1), c_\phi(u_6)$ respectively, as well as the $1$ color $c_\phi(f_2)$. This is a total of $10$, so we permute with the $11$th and $12$th available colors. We then permute $c_\phi(f_2)$. This must avoid the $6$ colors of $C$, the $2$ colors of $c_\psi(g_3), c_\psi(g_4)$, and the color $c_\psi(g_2)$. We permit $c_\phi(f_2)$ to permute with $c_\phi(u_1), c_\phi(u_6)$. This is a total of $9$ colors, so we permute with the $10$th available. We lastly permute $c_\phi(g_1), c_\phi(g_6)$. These must avoid the $6$ colors of $C$, the $3$ colors of $c_\psi(u_2), c_\psi(u_5)$, the one color $c_\phi(f_2)$, and in the case of $g_1$, the two colors $c_\psi(f_1), c_\psi(f_3)$, and in the the case of $g_6$, only the one color $c_\psi(f_1)$. This is a total of $12$ and $11$ colors respectively, so we permute $c_\phi(g_1)$ with the $13$th available, and $c_\phi(g_6)$ with the $13$th available. We have resolved all possible conflicts between $\phi$ and $\psi$, so we have a good coloring of $G$.

				\vspace{0.25cm}
				\noindent\textbf{Config. 3.1.4:} Suppose that $f_2, f_3$ are both interior edges.

				Color $\Ext(C) - C$ in $\psi$. We then have $|A_\psi(g_3)|, |A_\psi(g_4)| \ge 3$, $|A_\psi(f_1)| \ge 4$, $|A_\psi(g_2)|, |A_\psi(g_5)| \ge 5$, $|A_\psi(e_1)|, |A_\psi(e_6)| \ge 8$, $|A_\psi(e_3)|, |A_\psi(e_4)| \ge 9$, and $|A_\psi(e_2)|, |A_\psi(e_5)| \ge 11$. Therefore, by Hall's Theorem, we may color the full set of edges independently. We now color $\Int(C) - C$ in $\phi$. We have $|A_\phi(f_2)|, |A_\phi(f_3)| \ge 4$, $|A_\phi(g_1)|, |A_\phi(g_6)| \ge 5$, $|A_\phi(e_2)|, |A_\phi(e_5)| \ge 8$, $|A_\phi(e_3)|, |A_\phi(e_4)| \ge 10$, and $|A_\phi(e_1)|, |A_\phi(e_6)| \ge 11$. Therefore, we may also color all edges uniquely using Hall's Theorem. Finally, we permute $C$ in $\phi$ to be colored identically to $C$ under $\psi$, and then rejoin $\Int(C)$ and $\Ext(C)$. We now permute our colors in $\Int(C)$ and $\Ext(C)$ to resolve color conflicts. We first permute $c_\psi(g_2), c_\psi(g_5)$. These must avoid $6$ colors of $C$, $3$ colors of $c_\phi(u_1), c_\phi(u_6)$ respectively, and the $1$ color $c_\phi(f_2), c_\phi(f_3)$ respectively, for a total of $10$. We permute with the $11$th and $12$th available. We then permute $c_\phi(g_1), c_\phi(g_6)$. These must avoid the $6$ colors of $C$, the $3$ colors of $c_\psi(u_2), c_\psi(u_5)$ respectively, and the $1$ color $c_\psi(f_1)$, for a total of $10$. We permute with the $11$th and $12$th available colors. We must finally make $c_\phi(f_2), c_\phi(f_3)$ not conflict with $c_\psi(g_2), c_\psi(g_3)$. So, we permute $c_\phi(f_2), c_\phi(f_3)$. Note that $c_\phi(f_2), c_\phi(f_3)$ is good with $c_\psi(g_2), c_\psi(g_5)$, for we preivously permuted $c_\psi(g_2), c_\psi(g_5)$ to be good with $c_\phi(f_2), c_\phi(f_3)$, and if we permute $c_\phi(g_1), c_\phi(g_6)$ with $c_\phi(f_2), c_\phi(f_3)$ in the previous step, then these colors are again good with $c_\psi(g_2), c_\psi(g_5)$. Therefore, we must only avoid the $6$ colors of $C$, the two colors $c_\phi(g_1), c_\phi(g_6)$, and, for $c_\phi(f_2)$, the three colors $c_\psi(g_2), c_\psi(g_3), c_\psi(g_4)$, and for $c_\phi(f_3)$, the three colors $c_\psi(g_2), c_\psi(g_3), c_\psi(g_5)$.This is a total of $11$. So, we permute with the $12$th and $13$th respectively, giving a good coloring for all of $G$.

				\vspace{0.25cm}
				\noindent\textbf{Config. 3.1.5:} Suppose that $f_1$ and one of $f_2, f_3$ are interior edges.

				The cases are symmetric, so we suppose that $f_1, f_2$ are both interior edges. Color $\Ext(C) - C$ in $\psi$. We then have $|A_\psi(g_3)|, |A_\psi(g_4)| \ge 3$, $|A_\psi(f_3)| \ge 4$, $|A_\psi(g_2)|, |A_\psi(g_5)| \ge 5$, $|A_\psi(e_4)| \ge 6$, $|A_\psi(e_5)| \ge 8$, $|A_\psi(e_3)| \ge 9$, and $|A_\psi(e_1)|, |A_\psi(e_2)|, |A_\psi(e_6)| \ge 11$. By Hall's Theorem, we may color the resulting set of edges distinctly. Now, color $\Int(C) - C$ in $\phi$. We have $|A_\phi(f_1)|, |A_\phi(f_2)| \ge 4$, $|A_\phi(g_1)|, |A_\phi(g_6)| \ge 5$, $|A_\phi(e_1)|, |A_\phi(e_2)|, |A_\phi(e_6)| \ge 8$, $|A_\phi(e_3)| \ge 10$, $|A_\phi(e_5)| \ge 11$, and $|A_\phi(e_4)| \ge 13$. So, we may again color this set of edges distinctly by Hall's Theorem. Now, we first recolor $c_\psi(g_2), c_\psi(g_5)$. These must both avoid the $6$ colors of $C$, the $3$ colors of $c_\phi(u_1), c_\phi(u_6)$ respectively, the $1$ color $c_\phi(f_1)$ for $c_\psi(g_5)$, and the $2$ colors $c_\phi(f_1), c_\phi(f_2)$ for $c_\psi(g_5), c_\psi(g_2)$ respectively. This is a total of $10$ for $c_\psi(g_5)$ and $11$ for $c_\psi(g_2)$, so we color them in the $11$th and $13$th respectively. We then permute $c_\phi(g_1), c_\phi(g_6)$. These must avoid the $6$ colors of $C$, the $3$ colors of $c_\psi(u_2), c_\psi(u_5)$, and $c_\phi(g_6)$ must avoid $c_\psi(f_3)$. We also require that both avoid $c_\phi(f_1)$. This is $10$ colors for $c_\phi(g_1)$ and $11$ colors fo $c_\phi(g_6)$, so we permute them with the $11$th and $13$th available colors respectively. Finally, we permute $c_\phi(f_2)$. This must avoid the $6$ colors of $C$, the $2$ colors $c_\phi(g_1), c_\phi(g_6)$, and the $2$ colors $c_\psi(g_3), c_\psi(g_4)$. This is a total of $10$ colors, so we permute with the $11$th available. Note, we allow $c_\phi(f_2)$ to permute with $c_\phi(u_1) - c_\phi(g_1), c_\phi(u_6) - c_\phi(g_6)$. The argument for why this is permitted is identical to permuting $c_\phi(f_2), c_\phi(f_3)$ in Config. 3.1.4. We have now resolved all conflicts and achieved a good coloring for $G$.

				\vspace{0.25cm}
				\noindent\textbf{Config. 3.1.6:} Suppose that all $f_1, f_2, f_3$ are interior edges.

As usual, we color $\Ext(C) - C$ in $\psi$. We have $|A_\psi(g_3)|, A_\psi(g_4)| \ge 3$, $|A_\psi(g_2)|, |A_\psi(g_5)| \ge 5$, $|A_\psi(e_3)|, |A_\psi(e_4)| \ge 9$, $|A_\psi(e_1)|, |A_\psi(e_2)|, |A_\psi(e_5)|, |A_\psi(e_6)| \ge 11$, so we color this set of edges distinctly by Hall's Theorem. We then color $\Int(C) - C$ in $\phi$. We have $|A_\phi(f_i)| \ge 4$ for all $1 \le i \le 3$, $|A_\phi(g_1)|, |A_\phi(g_6)| \ge 5$, $|A_\phi(e_1)|, |A_\phi(e_2)|, |A_\phi(e_5)|, |A_\phi(e_6)| \ge 8$, and $|A_\phi(e_3)|, |A_\phi(e_4)| \ge 10$. We color $f_1$, and permit $e_3$ to be colored identically to $c_\phi(f_1)$ so that it has $10$ colors still available to it. We are then able to color the remaining set of edges distinctly using Hall's Theorem. We permute $C$ in $\phi$ to be colored identically to $C$ in $\psi$. We then join $\Ext(C)$ and $\Int(C)$ back together.We first permute $c_\psi(g_2), c_\psi(g_5)$. These must avoid the $6$ colors of $C$, the $3$ colors of $c_\phi(u_1), c_\phi(u_6)$ respectively, and the colors of the $2$ adjacent interior $f_i$, for a total of $11$ colors. We permute with the $12$th and $13$th available. We then permute $c_\phi(g_1), c_\phi(g_6)$. These must avoid the $6$ colors of $C$ and the $3$ colors of $c_\psi(u_2), c_\psi(u_5)$. We also require they avoid the $2$ colors $c_\psi(g_3), c_\psi(g_4)$, for a total of $11$ colors. We permute with the $12$th and $13$th. We finally permute $c_\psi(g_3), c_\psi(g_4)$. These must avoid the $6$ colors of $C$, the $2$ colors $c_\phi(f_2), c_\phi(f_3)$, and the $2$ colors $c_\psi(g_2), c_\psi(g_5)$. We allow them to permute with $c_\psi(u_2), c_\psi(u_5)$ otherwise, for $c_\phi(g_1), c_\phi(g_6)$ are good with $c_\psi(g_3), c_\psi(g_4)$. This is a total of $10$ colors, so we permute with the $11$th and $12$th available. We have resolved all color conflicts, and thus achieved a good coloring for $G$.

			\end{proof} 

                We may finally consider the second subcase.

                \vspace{0.25cm}
\noindent\textbf{Subcase 3.2:} Suppose that two interior edges of $C$ are incident to the same $4$-vertex.

Without loss of generality, suppose that $f_1, f_2, f_3$ are incident to $3$-vertices $v_1, v_3, v_5$. Let $g_1, g_2$ be interior edges incident to $v_2$. Let $g_3, g_4$ and $g_5, g_6$ be incident to $v_4, v_6$ respectively. We begin with the assumption that $f_1, f_2, f_3$ are all interior edges. We first color $\Int(C) - C$ in $\phi$. We have $|A_\phi(g_1)|, |A_\phi(g_2)| \ge 3$, $|A_\phi(f_1)|, |A_\phi(f_2)|, |A_\phi(f_3)| \ge 4$, $|A_\phi(e_1)|, |A_\phi(e_2)| \ge 6$, $|A_\phi(e_3)|, |A_\phi(e_6)| \ge 10$, and $|A_\phi(e_4)|,|A_\phi(e_5)| \ge 11$. We first color $g_1, g_2$, and we remove $c(g_1), c(g_2)$ from every remaining available color list. We then color $f_1, f_2, f_3$, noting that $f_1$ does not see $f_2, f_3$, or it has at least one additional available color, and thus we may color in order $f_3, f_2, f_1$. We also note that $e_1$ either does not see $f_3$, or it has more colors available to it. Therefore, we are left with $|A_\phi(e_2)| \ge 1$, $|A_\phi(e_1)| \ge 2$, $|A_\phi(e_3)|, |A_\phi(e_6)| \ge 5$, and $|A_\phi(e_4)|,|A_\phi(e_5)| \ge 6$. Thus, we may color the remaining edges of $C$ distinctly by Hall's Theorem. We then color $\Ext(C) - C$ in $\psi$. We have $|A_\psi(g_3)|, |A_\psi(g_4)|, |A_\psi(g_5)|, |A_\psi(g_6)| \ge 3$, $|A_\psi(e_3)|, |A_\psi(e_4)|, |A_\psi(e_5)|, |A_\psi(e_6)| \ge 9$, and $|A_\psi(e_1)|, |A_\psi(e_2)| \ge 13$. We first color $g_3, g_4, g_5, g_6$. We note that $g_6$ does not see $g_3, g_4$ or has at least one additional color, so that this coloring is possible. We are then left with $|A_\psi(e_3)|, |A_\psi(e_4)|, |A_\psi(e_5)|, |A_\psi(e_6)| \ge 5$ and $|A_\psi(e_1)|, |A_\psi(e_2)| \ge 8$, so we color the edges of $C$ distinctly by Hall's Theorem. We permute $C$ in $\Ext(C)$ to place $\Int(C)$ and $\Ext(C)$ back together. We then permute $c(g_3), c(g_4), c(g_5), c(g_6)$. Each must avoid the $6$ colors of $C$, plus the $3$ colors from $c(f_1), c(f_2), c(f_3)$. These are all conflicts possible, so we permute $c(g_3), c(g_4), c(g_5), c(g_6)$ with the $10$th, $11$th, $12$th, and $13$th available colors respectively. Note that if $c(g_6) = c(g_3), c(g_4)$, then it is automatically good.

Now, suppose that there are two $f_i$ in $\Int(C)$. We follow an identical coloring schematic for $\Int(C)$ as before. For $\Ext(C)$, we must split into two cases, one in which $f_3$ is in $\Ext(C)$, and one in which it is not. First, suppose that $f_3$ is in $\Ext(C)$. Then, we color $\Ext(C) - v_1v_2v_3$. In the resulting graph, we have $|A_\phi(e_3)|, |A_\phi(e_6)| \ge 4$ and $|A_\phi(e_1)|, |A_\phi(e_2)| \ge 10$. We remove $c(e_4), c(e_5)$ from the colors available to $e_1, e_2$ to get $|A_\phi(e_1)|, |A_\phi(e_2)| \ge 8$. We may color the resulting graph distinctly using Hall's Theorem. Otherwise, $f_3$ is not in $\Ext(C)$. The cases are symmetric, so we may suppose that $f_1$ is in $\Ext(C)$. We color $\Ext(C) - C$ as usual. We then have $|A_\phi(g_i)| \ge 3$, and $|A_\phi(f_{1})| \ge 4$. Now, $f_1, g_6$ either do not see $g_3, g_4$, or they have an additional available color for each that they do see. Thus, we may color $g_3, g_4, g_5, g_6, f_1$. In the resulting graph, we have $|A_\phi(e_6)| \ge 1$, $|A_\phi(e_3)|, |A_\phi(e_4)|, |A_\phi(e_5)| \ge 4$, $|A_\phi(e_1)| \ge 5$, and $|A_\phi(e_2)| \ge 8$. Thus, we may color the remaining edges distinctly by Hall's Theorem.

In both cases, we have a distinct set of colors on $C$ in both $\Ext(C)$ and $\Int(C)$. In the case of $\Int(C)$, we can also guarantee that $g_1, g_2, f_2, f_3$ are all colored distinctly from $C$. Therefore, permute $C$, and join $\Ext(C)$ and $\Int(C)$ back together at $C$. Then, we first $c(f_2), c(f_3)$. Thus must both avoid the $6$ colors of $C$, plus the at most $4$ colors of $c(g_3), c(g_4), c(g_5), c(g_6)$, and the color $c(f_1)$. That is a total of $11$, so we permute with the $12$th and $13$th available colors. We then permute $g_2, g_3$. These must avoid the $6$ colors of $C$ and $c(f_1), c(f_2), c(f_3)$. We also cannot swap $c(g_2), c(g_3)$ with each other. This is a total of $10$ colors, so we swap with the $11$th and $12$th. Thus, we have a good coloring for all of $C$.

Finally, suppose that there are two exterior edges $f_i$. These must be $f_1, f_2$, for if $f_1, f_3$ are exterior, or, symmetrically, $f_2, f_3$, then there is an adjacent vertex cut $v_1, v_2$, contradicting Lemma \ref{lem:no-adjacent-vertex-cut}. We use the exact same method to color $\Int(C)$. Note in this case that we can guarantee that $f_3, g_1, g_2$ and all colors of $C$ are distinct. Moving on to $\Ext(C)$ So, color $\Ext(C) - C$ in $\psi$. We first focus on the two disjoint sections $g_5, g_6, f_1$ and $g_3, g_4, f_2$. The edges of these two sets either do not see each other, or for each additional edge they see, have one more color available to them. Since each edge in each set has at least $3$ available colors, then we may color these external edges. After coloring these edges, we have $|A_\psi(e_6)|, |A_\psi(e_3)| \ge 3$, $|A_\psi(e_4)|, |A_\psi(e_5)| \ge 4$, and $|A_\psi(e_1)|, |A_\psi(e_2)| \ge 6$. So, we may color the edges of the cycle distinctly. Now, $C$ has all colors distinct in $\Int(C)$ and $\Ext(C)$, so we may join the two graphs together again after properly permuting the edges of $C$. In this graph, we first permute $c(f_3)$ on $f_3$, which must avoid the $6$ colors of $C$, plus the $4$ colors $c(g_3), c(g_4), c(g_5), c(g_6)$. This is a total of $10$, so we permute $c(f_3)$ with the $11$th available color. We then permute $c(g_1), c(g_2)$. These must avoid the $6$ colors of $C$, the $2$ colors $c(f_1), c(f_2)$, and the $1$ color of $c(f_1)$, for a total of $9$. So, we permute with the $10$th and $11$th available, for a good coloring of $G$. This is every case, resolving Subcase 3.2.
                    \end{proof}

        \vspace{0.25cm}
        \noindent\textbf{Case 4:} There are three internal edges incident to a $4$-vertex.
        \begin{proof}
            Let $f_1, f_2, f_3$ be incident to $e_1, e_3, e_5$ respectively. Let the pairs of edges $g_1$ and $g_2$, $g_3$ and $g_4$, and $g_5$ and $g_6$ be incident to $v_2$, $v_4$, $v_6$ respectively. Let the non-cycle endpoint of $g_i$ be $w_i$. As usual, we distinguish between two subcases.
            
            \vspace{0.25cm}
            \noindent\textbf{Subcase 4.1:} The three internal edges are all incident to distinct $4$-vertices

            \begin{proof}[Subproof]
                Without loss of generality, suppose that $g_1, g_3, g_5$ are the interior edges. We split between the case where there is one internal $f_i$ and no internal $f_i$. If there were two or more internal $f_i$, we could symmetrically choose to perform our proof on $\Ext(C)$.

								\vspace{0.25cm}
							  \noindent\textbf{Config. 4.1.1:} Suppose that there are no interior edges of $C$.

                Note that $w_1, w_3, w_5$ cannot all be adjacent to each other, for then we contradict Lemma \ref{no-3-cycles}. So, suppose without loss of generality that $w_1, w_5$ are not adjacent. We may also guarantee that $w_2, w_6$ are not adjacent, for otherwise $f_1$ is contained in the separating $5$-cycle $w_2w_6v_6v_1v_2w_2$. So, attach the edge $w_2w_6$ to the graph $\Ext(C) -C$, and color the resulting graph in $\psi$. Place $c(w_2w_6)$ on $g_2, g_6$. We then have $|A_\psi(f_1)|, |A_\psi(f_2)|, |A_\psi(f_3)| \ge 3$, $|A_\psi(g_4)| \ge 5$, and $|A_\psi(e_i)| \ge 7$ for all $1 \le i \le 6$. Therefore, since each $e_i$ sees $c(w_2w_6)$, we have $|A_\psi(e_i) \cup A_\psi(e_j)| \le 12$, forcing $|A_\psi(e_i) \cap A_\psi(e_j)| \ge 2$. Then, we color $e_1, e_4$ and $e_2, e_5$ in identical colors. We are left with $|A_\psi(f_i)| \ge 1$, $|A_\psi(g_4)| \ge 3$, and $|A_\psi(e_3)|, |A_\psi(e_6)| \ge 5$. Note that we permit $g_4$ to be colored identically to $g_2, g_6$. We then color $f_1, f_2, f_3$. Either these edges do not see each other, or they have additional colors available. We likewise permit $g_4$ to be colored identically to $f_1$, and $e_6$ to be colored identically to $f_2$. Thus, we have $|A_\psi(g_4)| \ge 1$, $|A_\psi(e_3)| \ge 2$, and $|A_\psi(e_6)| \ge 3$. We use Hall's Theorem to color the remaining set of edges distinctly. Note that we will permute the exterior first, so we never encounter difficulty with the possibly identical colors on $f_1, g_4$ or $f_2, e_6$.

                We now color attach the edge $w_1w_5$ to $\Int(C) - C$, and color the resulting graph in $\phi$. Note by assumption that $w_1w_5$ does not exist. Then, apply the color $c(w_1w_5)$ to $g_1, g_5$. We require that $g_3$ not be colored in $c(w_1w_5)$, so we have $|A_\phi(g_3)| \ge 4$ and $|A_\phi(e_i)| \ge 10$, for all $1 \le i \le 6$. Therefore, $|A_\phi(e_1) \cap A_\phi(e_4)| \ge 2$, and $|A_\phi(e_2) \cap A_\phi(e_5)| \ge 2$, so we give the pairs $e_1, e_4$ and $e_2, e_5$ the same colors. We then have $|A_\phi(g_4)| \ge 2$ and $|A_\phi(e_3)|, |A_\phi(e_6)| \ge 8$. We color the rest of the graph distinctly using Hall's Theorem. Notice by our requirements that all edges are colored distinctly except the edge pairs $e_1, e_4$ and $e_2, e_5$. Note also that $C$ has the same arrangement of colors in $\Int(C)$ as in $\Ext(C)$, so permute the edges of $C$ to have the same colors in $\Int(C)$ as $\Ext(C)$, and join $\Int(C)$ and $\Ext(C)$ back together.
                
                We first permute the exterior of the graph, and we split into two cases. Suppose that $c(g_2) = c(g_4) = c(g_6)$ as well. Then, if $c(g_2)$ is bad on $g_2, g_4, g_6$, we permute it. It must avoid $4$ colors on $C$, plus at most $8$ colors on $c(w_1), c(w_3), c(w_5)$, for a total of $12$. We permute with the $13$th available color. We then permute the interior. Beginning with $c(g_1) = c(g_5)$, these must avoid the $4$ colors of $C$, the $5$ colors of $c(w_2), c(w_6)$, and the $3$ colors $c(f_1), c(f_2), c(f_3)$, for a total of $12$. We permute with the $13$th available. We then permute $c(g_3)$, which must avoid the $4$ colors of $C$, the $3$ colors of $c(w_4)$, the $2$ colors of $c(f_2), c(f_3)$, and the $1$ color $c(g_1) = c(g_5)$, for a total of $10$. Note that $c(g_3)$ may permute with the other, non-$c(g_1)$ colors of $c(w_1), c(w_5)$, for $c(g_2) = c(g_4) = c(g_6)$ is good with $c(g_1)$. Therefore, it must avoid at most $10$ colors, so we permute with the $11$th available. 
                
                Otherwise, $c(g_2) \neq c(g_4)$. In this case, $c(g_2) = c(g_6)$ must avoid the $4$ colors of $C$ and the $5$ colors of $c(w_1), c(w_5)$, for a total of $9$. We also require that it avoid $c(g_3)$. This is $10$ colors, so we permute with the $11$th available. We then permute $c(g_4)$, if it is bad. It must avoid the $4$ colors of $C$, the $3$ colors of $c(w_3)$, and the $1$ color $c(g_2)$, for a total of $8$. We also require that it avoid the color $c(g_1) = c(g_5)$. We permute with the $10$th available. We then permute the interior. We begin with $c(g_1) = c(g_5)$, which must avoid the $4$ colors of $C$, the $5$ colors of $c(w_2), c(w_6)$, and the $3$ colors $c(f_1), c(f_2), c(f_3)$, for a total of $12$. We permute with the $13$th available. We finally permute $c(g_3)$. We first mention that $c(g_3)$ is good with $c(g_2), c(g_4), c(g_6)$, for either we have permuted $c(g_2), c(g_4), c(g_6)$ to be good with $c(g_3)$, or $c(g_3)$ was just swapped in the previous step, and is thus also a good color for those sets of edges. Thus, $c(g_3)$ may be permuted with the non-$c(g_1)$ colors of $c(w_1), c(w_3), c(w_5)$. Thus, $c(g_3)$ must avoid the $4$ colors of $C$, the $3$ colors of $c(w_4)$, the $2$ colors of $c(f_2), c(f_3)$, and the $1$ color $c(g_1) = c(g_5)$, for a total of $10$. We permute with the $11$th available.

								\vspace{0.25cm}
								\noindent\textbf{Config. 4.1.2:} There is one interior $f_i$.

                Suppose without loss of generality, that $f_1$ is an interior edge, and that $f_2, f_3$ are exterior. Construct a new graph $H$ from $\Ext(C) - C$ by adding the edge $w_2w_4$, and color this graph in $\psi$. Note that $w_2w_4$ does not already exist, for then $w_2w_4v_4v_3v_2w_2$ is a separating $5$-cycle, violating Lemma \ref{lem:no-sep-5-cycle}. For this reason, as well as Lemma \ref{lem:no-sep-4-cycle} and Lemma \ref{no-3-cycles}, $g_2$ does not see $g_4$. Therefore, place $c(w_2w_4)$ on $g_2, g_4$. Then, $|A_\psi(f_2)|, |A_\psi(f_3)| \ge 3$, $|A_\psi(g_6)| \ge 4$, $|A_\psi(e_2)|, |A_\psi(e_3)|, |A_\psi(e_4)|, |A_\psi(e_5)| \ge 7$, and $|A_\psi(e_1)|, |A_\psi(e_6)| \ge 10$. So, $|A_\psi(e_2) \cap A_\psi(e_5)| \ge 1$, since $|A_\psi(e_2)|, |A_\psi(e_5)| \ge 7$. Moreover, $|A_\psi(e_3)| \ge 7$, $|A_\psi(e_6)| \ge 10$, and $|A_\psi(e_3)\cup A_\psi(e_6)| \le 12$, since they both see $c(w_2w_4)$. Therefore, $|A_\psi(e_3)\cap A_\psi(e_6)| \ge 5$. Thus, by Hall's Theorem, we may give a distinct color to every edge, except for the specified pairs of edges $e_3, e_6$ and $e_2, e_5$, which we color identically. We once again construct a new graph from $\Int(C) - C$ by adding the edge $w_1w_5$. Following the same argument as above, it is possible to add this edge, and, moreover, we may guarantee that $g_1, g_5$ do not see each other. Color the resulting graph in $\phi$, and then place $c(w_1w_5)$ on $g_1, g_5$. We now have $|A_\phi(f_1)| \ge 3$, $|A_\phi(g_3)| \ge 4$, $|A_\phi(e_1)|, |A_\phi(e_6)| \ge 7$, and $|A_\phi(e_2)|, |A_\phi(e_3)|, |A_\phi(e_4)|, |A_\phi(e_5)| \ge 10$. So, $|A_\phi(e_6) \cap A_\phi(e_3)| \ge 1$, and likewise $|A_\phi(e_2) \cap A_\phi(e_5)| \ge 2$. So, we may use Hall's Theorem to color all edges distinctly, except for the specified pairs $e_3, e_6$ and $e_2, e_5$.  Note that $C$ has the same arrangement of colors in $\Int(C)$ as in $\Ext(C)$.

                Since $C$ has the same arrangement of colors in $\Int(C)$ and $\Ext(C)$, we may permute the colors of $C$ in $\Int(C)$, and then rejoin the two graphs to form all of $G$. We first permute $c(g_2) = c(g_4), c(g_6)$ in $\Ext(C)$ to be good with $\Int(C)$. We then permute $c(g_1) = c(g_5), c(g_3)$ to be good with $\Ext(C)$. In the process, we make every $f_i$ good. So, we first permute $c(g_2) = c(g_4)$. This must avoid the $4$ colors of $C$, the $6$ colors of $c(w_1) \cup c(w_3)$, and the $1$ colors of $c(f_1)$, for a total of $11$. We permute with the $12$th available color. We now permute $c(g_6)$ on $g_6$. This must avoid the $4$ colors of $C$, the $3$ colors on $c(w_5)$,  the $1$ color on $c(f_1)$, the $1$ color $c(g_2)$, and the $1$ color $c(g_1) = c(g_3)$, for a total of  $10$. We permute with the $11$th available color. Note now that $c(f_1)$ is automatically good. We now permute $c(g_1) = c(g_5)$. This must avoid the $4$ colors of $C$, the $6$ colors of $c(w_2), c(w_6)$, and the two colors $c(f_2), c(f_3)$, for a total of $12$. We permute with the $13$th available. We finally permute $c(g_3)$ on $g_3$. Recall that this is good with $c(g_4) = c(g_2)$, even if we had chosen to permute $c(g_1)$ with $c(g_3)$. So, we may permute with the colors of the edges who have the endpoint $w_1$, except for $c(g_1)$. We also may permute $c(g_3)$ with its own set of colors $c(w_3)$. Thus, we must avoid the $4$ colors of $C$, the $3$ colors of $c(w_4)$, the $1$ color $c(g_2) = c(g_4)$, the $2$ additional colors of $c(w_5)$, and the $2$ colors $c(f_2), c(f_3)$. This is a total of $12$, so we permute with the $13$th available color. We now have a good coloring for all of $G$, for there are no color conflicts between the $g_i$, and there are no color conflicts between the $g_i$ and $f_i$.
            \end{proof}

        \vspace{0.25cm}
        \noindent\textbf{Subcase 4.2:} Two internal edges are adjacent to the same $4$-vertex.
        
Without loss of generality, suppose that $g_1, g_2, g_3$ are internal edges. Thus, $g_1, g_2$ are adjacent to the same $4$-vertex, $v_2$. Note that $\Int(C)$ may contain at most one of the $f_i$, for $\Int(C)$ and $\Ext(C)$ are symmetric. We first color $\Ext(C)$. We are not concerned about repeat colors among the $f_i, g_i$, for we will permute $g_4$ once first, and then permute $\Int(C)$. Thus, color $\Ext(C)$ using minimality. Now, color $\Int(C) - C$ in $\phi$. We simply color the case in which $C$ If $C$ has no internal $f_i$, then we have $|A_\phi(g_1)|, |A_\phi(g_2)| \ge 3$, $|A_\phi(g_3)| \ge 5$, $|A_\phi(e_1)|, |A_\phi(e_2)| \ge 9$, $|A_\phi(e_3)|, |A_\phi(e_4)| \ge 11$, and $|A_\phi(e_5)|, |A_\phi(e_6)| \ge 13$. Therefore, $|A_\phi(e_1) \cap A_\phi(e_4)| \ge 7$, $|A_\phi(e_2) \cap A_\phi(e_5)| \ge 9$, and $|A_\phi(e_3) \cap A_\phi(e_6)| \ge 11$. If there is an internal $f_i$, then we therefore have $|A_\phi(f_i)| \ge 4$, and $|A_\phi(e_1) \cap A_\phi(e_4)| \ge 4$, $|A_\phi(e_2) \cap A_\phi(e_5)| \ge 6$, and $|A_\phi(e_3) \cap A_\phi(e_6)| \ge 8$. Therefore, we may always color in order $g_1$, $g_2$, $f_i$, then the pair $e_1, e_4$, the pair $e_2, e_5$, and the pair $e_3, e_6$. Otherwise, if $e_i, e_j$ can not be colored in the same color, then we may color them last. We have $|A_\phi(e_1)|, |A_\phi(e_2)| \ge 6$, $|A_\phi(e_3)|, |A_\phi(e_4)| \ge 8$, and $|A_\phi(e_5)|, |A_\phi(e_6)| \ge 10$, so we may color the remaining set of edges distinctly, if necessary.

Now, permute $C$ in $\Int(C)$ so that we may join $\Int(C)$ and $\Ext(C)$ back together. Recall we have guaranteed by our coloring that $C$ has the same arrangement of colors in $\Int(C)$ and $\Ext(C)$. Now, we first permute $c(g_4)$ on $g_4$. We cannot permute with the $6$ colors of $C$, nor the $3$ colors of $c(w_3)$, and we also avoid the $3$ colors $c(g_1), c(g_2)$, and $c(f_i)$ for $f_i$ internal to $C$. This is a total of $12$ colors, so we permute with the $13$th available. We then permute $c(g_3)$ on $g_3$. We must avoid the $6$ colors of $C$, the $3$ colors of $c(w_4)$, and the at most $2$ colors of the neighboring $c(f_j)$, for a total of $11$. We permute with the $12$th color. We then permute $c(g_1), c(g_2)$. These must avoid the $6$ colors of $C$, the colors $c(f_1), c(f_2)$, and the $1$ color $c(g_3)$. Note we may permute with the remaining colors of $c(w_3)$, for $c(g_1), c(g_2)$ are either colored to be good with $c(g_4)$, or were swapped in the last step with $c(g_3)$, which is also good with $c(g_4)$. Therefore, these must avoid at most $9$ colors, so we permute them with the $10$th and $11$. We lastly permute the one interior $c(f_i)$. This must avoid the $6$ colors of $C$, the $3$ colors $c(g_1), c(g_2), c(g_3)$, and the at most $3$ neighboring edges of $\Ext(C)$, for a total of $12$. Note that we may permute with the non-$c(g_3)$ colors of $c(w_3)$, for either we have colored $c(g_4)$ to be good with $c(f_i)$, or $c(f_i)$ was swapped with the previous color on $g_1, g_2, g_3$, all of which are also good with $c(g_4)$. Thus, $c(f_i)$ must avoid at most $12$ colors to permute with, so we permute with the $13$th available. We have thereby achieved a good coloring for all of $G$.
        \end{proof}
        These constitute all possible cases, so there do not exist any separating $6$-cycles in $G$.
    \end{proof}

\end{document}